\documentclass[11pt]{article}
\usepackage{mathrsfs}
\usepackage{mathrsfs}
\usepackage{mathrsfs}
\usepackage{amsfonts}
\usepackage{amssymb}
\usepackage{amsmath}
\usepackage{amsthm}
\usepackage{color}
\usepackage{appendix}
\theoremstyle{plain}
\newtheorem{thm}{Theorem}[section]

\newtheorem{lem}[thm]{Lemma}
\newtheorem{prop}[thm]{Proposition}

\newtheorem{rem}[thm]{Remark}
\newtheorem{prob}[thm]{Problem}

\theoremstyle{definition}
\newtheorem{defn}{Definition}[section]

\newtheorem{case}{Case}

\newtheorem*{Th1.1}{Theorem 1.1}
\newtheorem*{Th1.2}{Theorem 1.2}

\newtheorem*{case1}{Case 1}
\newtheorem*{case2}{Case 2}
\newtheorem*{case1.1}{Case 1.1}
\newtheorem*{case2.1}{Case 2.1}
\newtheorem*{case1.2}{Case 1.2}
\newtheorem*{case2.2}{Case 2.2}

\def\no{\noindent}
\def\bs{\bigskip}

\begin{document}
\title{{\bf {The norms for symmetric and antisymmetric tensor products of the weighted shift operators }}
}
\author{{\normalsize Xiance Tian,    Penghui Wang  and   Zeyou Zhu } \\
{ School of Mathematics, Shandong University,} {\normalsize Jinan, 250100, China} }

\date{}
\maketitle

\begin{abstract}
In the present paper,  we study the norms for symmetric and antisymmetric tensor products of weighted shift operators. By proving  that for $n\geq 2$,
	$$\|S_{\alpha}^{l_1}\odot\cdots \odot S_{\alpha}^{l_k}\odot S_{\alpha}^{*l_{k+1}}\odot\cdots \odot S_{\alpha}^{*l_{n}}\|  =\mathop{\prod}_{i=1}^n\left \| S_{\alpha}^{{l_{i}}}\right\|,  \text{ for any} \ (l_1,l_2\cdots l_n)\in\mathbb N^n$$ if and only if the weight satisfies the regularity condition, we partially solve  \cite[Problem 6 and Problem 7]{GA}.
It will be seen that many weighted shift operators on function spaces, including weighted Bergman shift, Hardy shift, etc, satisfy the regularity condition. Moreover, at the end of the paper, we solve \cite[Problem 1 and Problem 2]{GA}. 
\end{abstract}

\bs


\bs

\no{\bf Key Words:} Symmetric tensor products,  antisymmetric tensor products,  weighted shift operators.

\numberwithin{equation}{section}
\newtheorem{theorem}{Theorem}[section]
\newtheorem{lemma}[theorem]{Lemma}
\newtheorem{proposition}[theorem]{Proposition}
\newtheorem{corollary}[theorem]{Corollary}

\section{Introduction}
\indent\indent

Symmetric tensor products and antisymmetric tensor products are essential in various fields. 
For instance, symmetric tensors  are foundational to general relativity \cite{in4}. Additionally, symmetric tensor products play  significant roles in multilinear algebra \cite{in5}, representation theory \cite{in6} and statistics \cite{in13}.
Moreover, decomposing a symmetric tensor into a minimal linear combination of tensor powers of
the same vector arises in mobile communications, machine learning, factor analysis of $k$-way
 arrays, biomedical engineering, psychometrics, and chemometrics \cite{in7,in8,in9,in10,in11}.
 Furthermore, the symmetric part of a quantum geometric tensor has been utilized as a tool to detect quantum phase transitions in PT-symmetric quantum mechanics \cite{in18}.
In particular, symmetric tensor products of operators are crucial in the study of non-interacting systems \cite{in14} and are actively explored within the physics community \cite{in15,in16,in17}.

In the present paper, we will study the norm of symmetric and antisymmetric tensor product of operators. 
Let $\mathcal{H}$ be a Hilbert space and $\mathcal{B}(\mathcal{H})$ be the space of bounded linear operators on $\mathcal{H}$. Let $\mathcal{H}^{\otimes n}$ be the full tensor space of $\mathcal{H}$, $\mathcal{H}^{\odot n}$ be the symmetric tensor space of $\mathcal{H}$ and $\mathcal{H}^{\wedge n}$ be the antisymmetric tensor space of $\mathcal{H}$, as introduced in \cite[Page 106]{QU}. For $\mathbf{v}_1, \mathbf{v}_2, \cdots, \mathbf{v}_n \in \mathcal{H}$, let $\mathbf{v}_1 \odot \mathbf{v}_2 \odot \cdots \odot \mathbf{v}_n$ be the symmetric tensor and $\mathbf{v}_1 \wedge \mathbf{v}_2 \wedge \cdots \wedge \mathbf{v}_n$ be the antisymmetric tensor of $\mathbf{v}_1, \mathbf{v}_2, \cdots, \mathbf{v}_n$.
For $\mathbb{A}_{1},\mathbb{A}_{2},\cdots,\mathbb{A}_{n} \in \mathcal{B}(\mathcal{H})$, we define $\mathbb{A}_{1}\otimes\mathbb{A}_{2}\otimes\cdots\otimes\mathbb{A}_{n}$ on $\mathcal{H}^{\otimes n}$ as
$$
\left(\mathbb{A}_{1}\otimes\mathbb{A}_{2}\otimes\cdots\otimes\mathbb{A}_{n}\right)\left(\mathbf{v}_1 \otimes \mathbf{v}_2 \otimes \cdots \otimes \mathbf{v}_n\right)=\mathbb{A}_{1} \mathbf{v}_1 \otimes \mathbb{A}_{2} \mathbf{v}_2 \otimes \cdots \otimes \mathbb{A}_{n} \mathbf{v}_n.
$$
Then it is not difficult to see that $\mathcal{H}^{\odot n}$ and $\mathcal{H}^{\wedge n}$ are invariant under the action of
$$
\mathrm{S}_n\left(\mathbb{A}_{1},\mathbb{A}_{2},\cdots,\mathbb{A}_{n}\right)=\frac{1}{n!} \sum_{\pi \in \Sigma_n}\left(\mathbb{A}_{\pi(1)} \otimes \mathbb{A}_{\pi(2)} \otimes \cdots \otimes \mathbb{A}_{\pi(n)}\right) \in \mathcal{B}\left(\mathcal{H}^{\otimes n}\right),
$$
where $\Sigma_{n}$ is the group of permutations of $\left \{ 1,2,\cdots ,n \right \} $. Set
$$\mathbb{A}_{1} \odot \mathbb{A}_{2} \odot \cdots \odot \mathbb{A}_{n}=\mathrm{S}_{n}\left(\mathbb{A}_{1}, \mathbb{A}_{2}, \cdots, \mathbb{A}_{n}\right)|_{\mathcal{H}^{\odot n}},$$
and
$$\mathbb{A}_{1} \wedge \mathbb{A}_{2} \wedge \cdots \wedge \mathbb{A}_{n}=\mathrm{S}_{n}\left(\mathbb{A}_{1}, \mathbb{A}_{2}, \cdots, \mathbb{A}_{n}\right)|_{\mathcal{H}^{\wedge n}}.$$

By \cite[Page 162]{BP} and \cite[(3.8.17)]{BSI}, $$\sigma(\mathbb{A}_1\otimes \mathbb{A}_2\otimes\cdots\otimes\mathbb A_n)=\sigma(\mathbb{A}_1)\cdots\sigma(\mathbb{A}_n)$$
 and
$$
\|\mathbb{A}_1\otimes \mathbb{A}_2\otimes\cdots\otimes\mathbb A_n\|=\prod\limits_{i=1}^n \|\mathbb A_i\|.
$$
However the spectra and norms of symmetric tensor products and antisymmetric tensor products of bounded operators are much more complicated. For example, in \cite[Proposition 7.2]{GA}, Garcia, O'Loughlin and Yu estimated the norms of symmetric tensor product of diagonal operators $\mathbb A_1$ and $\mathbb A_2$, and the sharpness of the following estimates was obtained: 
$$
(\sqrt{2}-1)\|\mathbb A_1\|\cdot\|\mathbb A_2\|\leq\left\|\mathbb A_1\odot \mathbb A_2\right\|\leq\|\mathbb A_1\|\cdot\|\mathbb A_2\|;
$$
moreover, they raised a series of problems aimed at exploring the norms and spectra of symmetric and antisymmetric tensor product of operators. To continue, let $\{e_i\}_{i=0}^\infty$ be an orthonormal basis(ONB for short) of a complex Hilbert space $\cal H$.
 {\it In what follows, we always assume that the weight  $\alpha=\{\alpha_i\}_{i=0}^\infty$ is  convergent}. 
Let \( S_{\alpha} \) be a weighted shift operator on \(\mathcal{H}\) with weights \(\alpha\), defined by 
    \[
    S_{\alpha} e_{i} = \alpha_{i} e_{i+1} \quad \text{for } i \ge 0,
    \]
    which has been extensively studied in \cite{Shields}. Let \( S_{\alpha}^{*} \) denote its adjoint operator.  \emph{In particular, \( S \) is the forward shift operator with the constant weight \(\alpha_i = 1\) for \( i \in \mathbb{N} \).}
    
    The following problems were raised in \cite{GA}.

\begin{prob}\label{prob1}(\cite[Problem 6]{GA})
	Identify the norm and spectrum of arbitrary symmetric or antisymmetric tensor products of $S$ and $S^{*}$(for example, consider $S^2 \odot S \odot S^{*3} $ and $S^2 \wedge S \wedge S^{*3} $).
\end{prob}
\begin{prob}\label{prob2}(\cite[Problem 7]{GA})
	Describe the norm and spectrum of $S_{\alpha } \odot S_{\alpha }^{*}$ and $S_{\alpha } \wedge S_{\alpha }^{*}$, in which $S_{\alpha }$ is a weighted shift operator. What can be said if more factors are included?
\end{prob}

We note that Yang and Zhang \cite{YZ} addressed spectral-related aspects of
Problem \ref{prob2} very recently. In this paper, we partially solve these problems. To clarify our main results, we need the assumption of the so-called regularity on the weight.

\begin{defn}\label{def:regularity}
(regularity): The weight $\alpha$ is called to be regular if 
 $$\lim\limits_{i \to \infty } \left | \alpha_{i}  \right | \ge \left | \alpha_{m}  \right |, \text{ for any } m \in \mathbb{N}. $$ 
\end{defn}
The following theorem is the main result in the present paper, which partially gives an affirmative answer to Problem 1.1 and Problem 1.2.
\begin{thm}\label{thm1.1}
 Let $S_\alpha$ be a weighted shift operator with the convergent weight. Then
 the following statements are equivalent.
\begin{itemize}
     \item[(1)] $\left \{ \alpha _{i}  \right \} _{i=0}^{\infty }$ satisfies the regularity condition;
     \item[(2)]
 for every $ n \ge 2$ and every fixed $\left ( l_{1},  l_{2},\cdots,l_{n} \right )\in \mathbb{N}^{n}$,
	$$\|S_{\alpha}^{l_1}\odot\cdots \odot S_{\alpha}^{l_k}\odot S_{\alpha}^{*l_{k+1}}\odot\cdots \odot S_{\alpha}^{*l_{n}}\|  =\mathop{\prod}_{i=1}^n\left \| S_{\alpha}^{{l_{i}}}\right\|;$$
     \item[(3)]for every $ n \ge 2$ and every fixed $\left ( l_{1},  l_{2},\cdots,l_{n} \right )\in \mathbb{N}^{n}$,
	$$\|S_{\alpha}^{l_1}\wedge\cdots \wedge S_{\alpha}^{l_k}\wedge S_{\alpha}^{*l_{k+1}}\wedge\cdots \wedge S_{\alpha}^{*l_{n}}\|=\mathop{\prod}_{i=1}^n\left \| S_{\alpha}^{{l_{i}}}\right\|.$$
\end{itemize}
In this case, set $\lambda =\lim\limits_{i \to \infty}\left | \alpha_{i}  \right |,$ $$\|S_{\alpha}^{l_1}\odot\cdots \odot S_{\alpha}^{l_k}\odot S_{\alpha}^{*l_{k+1}}\odot\cdots \odot S_{\alpha}^{*l_{n}}\|=\|S_{\alpha}^{l_1}\wedge\cdots \wedge S_{\alpha}^{l_k}\wedge S_{\alpha}^{*l_{k+1}}\wedge\cdots \wedge S_{\alpha}^{*l_{n}}\|=\lambda ^{  l_{1}  +  l_{2}  + \cdots + l_{n} }.$$
\end{thm}
\begin{rem}
\begin{itemize}
\item[1.] As a consequence of the above theorem, for $(l_1,\cdots,l_n)\in \mathbb N^n,$
$$
\|S^{l_1}\odot\cdots \odot S^{l_k}\odot S^{*l_{k+1}}\odot\cdots \odot S^{*l_{n}}\|=\|S^{l_1}\wedge\cdots \wedge S^{l_k}\wedge S^{*l_{k+1}}\wedge\cdots \wedge S^{*l_{n}}\|=1,
$$
which partially solves Problem \ref{prob1}.
\item[2.] It is evident that the shift operators on nearly every classical function space, including the Hardy shift and  weighted Bergman shift, are regular.
\end{itemize}
\end{rem}

This paper is organized as follows. In Section 2, we provide some preparations to facilitate the proof of Theorem \ref{thm1.1}.
 Section 3 is devoted to the proof of our main result, Theorem \ref{thm1.1}.  Finally, in Section 4, we solve \cite[Problem 1 and Problem 2]{GA}.

\section{Preliminaries}

In this section, we  introduce some notations and preliminaries.
For simplicity, set
\begin{eqnarray}	S_{\alpha,t}= \begin{cases}
	S_{\alpha}^{t},& t > 0, \\
I,& t =0, \\
	S_{\alpha}^{*\left | t \right | },& t < 0.
\end{cases}\end{eqnarray}
For example, $S_{\alpha,-2}=S_\alpha^{*2}$. Particularly, 
\begin{eqnarray}
S_{t}= \begin{cases}
	S^{t},& t > 0, \\
I,& t =0, \\
	S^{*\left | t \right | },& t < 0.
\end{cases}
\end{eqnarray}
Moreover, for $i \in \mathbb{N}$, $t$ $\in \mathbb{Z}$,  write
\begin{equation}\label{eq1}
	\beta_{i,t} = \begin{cases}
		\alpha _{i}\alpha _{i+1}\cdots\alpha _{i+t-1},& t > 0, \\
1,& t =0, \\
		\overline{\alpha_{i-1}\alpha_{i-2} \cdots\alpha_{i+t}},& t < 0,
	\end{cases}
\end{equation}
and
\begin{equation}\label{eq71}\Gamma_{i,t} = \begin{cases}
	\min \left \{ \left | \alpha _{i} \right |, \cdots,\left | \alpha _{i+t-1} \right | \right \}  ,& t > 0, \\
1 ,& t= 0, \\
	\min \left \{ \left | \alpha_{i-1} \right |, \cdots,\left | \alpha_{i+t} \right | \right \},& t < 0.
\end{cases}\end{equation}

For $\mathfrak{i}=(i_1,i_2,\cdots, i_n)\in\mathbb Z^n$, write $$|\mathfrak{i}|=|i_1|+\cdots+|i_n|.$$ Define an equivalence relation $\sim_{\Sigma_n}$ on $\mathbb Z^n$ by 
$$\mathfrak{i}\sim_{\Sigma_n}\mathfrak{j}$$ if and only if there exists a permutation $\pi\in\Sigma_n$ such that 
$$
i_k=j_{\pi (k)}, \quad k=1,\cdots, n.
$$
Denote $\mathbb N^n/\Sigma_n$ and $\mathbb Z^n/\Sigma_n$ to be the quotient spaces and $[i_1,\cdots,i_n] \in \mathbb Z^n/\Sigma_n $ \text{($[\mathfrak{i}]$ for short)} is the coset of $\mathfrak{i}=(i_1,\cdots,i_n)$.   For $\mathfrak{i}\in\mathbb Z^n$, $\pi\in\Sigma_n$, set 
\begin{eqnarray}
\mathfrak{i}_\pi=(i_{\pi(1)},\cdots,i_{\pi(n)}).
\end{eqnarray}

 The following lemma comes from the definition of symmetric tensor product of operators and some  direct calculations easily, and we omit its proof. {\it For convenience, when $i<0$, set $e_i=0$ and $\alpha_i=0$.}
\begin{lem}\label{lem2.1}
	For $\mathfrak{i}\in\mathbb N^n$,
	$$\begin{aligned}
		& \left ( S_{\alpha,{l_{1}}}\odot S_{\alpha,{l_{2}}}\odot \cdots \odot S_{\alpha,{l_{n}}}  \right ) \left ( e_{i_{1} }\odot e_{i_{2} }\odot \cdots \odot e_{i_{n}}    \right )
		\\ &=\frac{1}{n!} \sum\limits_{\pi  \in \Sigma _{n} }\beta_{i_{1},l_{\pi  (1)}} \beta_{i_{2},l_{\pi  (2)}}\cdots\beta_{i_{n},l_{\pi  (n)}} \left (  e_{i_{1}+l_{\pi  (1)} }\odot e_{i_{2}+l_{\pi  (2)} }\odot \cdots \odot e_{i_{n} +l_{\pi  (n)}} \right ).
	\end{aligned}$$	
\end{lem}

To simplify the notation, for $\mathfrak{i}\in \mathbb N^n$, set 
\begin{eqnarray}
e_{\mathfrak{i}}=e_{i_{1} }\odot e_{i_{2} }\odot \cdots \odot e_{i_{n}}, 
\end{eqnarray}
and it is easy to see that for any $\pi\in\Sigma_n$, $e_{\mathfrak{i}}=e_{\mathfrak{i}_\pi}$ and write
\begin{eqnarray}
e_{[\mathfrak{i}]}=e_{\mathfrak{i}}.
\end{eqnarray}
In what follows, we will always fix the multiple index $\mathfrak{l}=(l_1,l_2,\cdots,l_n)\in \mathbb{Z}^n$ and assume that
\begin{eqnarray}\label{eq2.3}
[l_{1},  l_{2},\cdots,l_{n}]=[\underbrace{{\tilde{l}_{1}},\cdots,{\tilde{l}_{1}} }_{n_{1}} ,\cdots,\underbrace{{\tilde{l}_{k}},\cdots,{\tilde{l}_{k}} }_{n_{k}}] , \quad \mathrm{where} \ {\tilde{l}_{i}} \ne  {\tilde{l}_{j}} \text{ for }  i\not=j,
\end{eqnarray}
and write
\begin{eqnarray}
S_{\mathfrak{l}}=l_1+l_2+\cdots+l_n \text{ and }\mathscr{M}=\frac{n!}{n_{1}!\cdots n_{k}! }.
\end{eqnarray}
Furthermore, for  $\mathfrak{i}\in \mathbb{N}^{n},$
to simply the notations let
\begin{equation}\label{eq2}
R_{\mathfrak{i}}=\left\{[\mathfrak{j}]\in\mathbb N^n/\Sigma_n:  \exists \, \pi \in \Sigma_{n}~s.t.~  [\mathfrak{j}+\mathfrak{l}_\pi]=[\mathfrak{i}]\right\},
\end{equation}
and for a finite set $\mathcal{S}$,  $^{\# } \mathcal{S} $  is defined to be the cardinality of  $\mathcal{S}$. 
Then it is easy to see that
\begin{equation}\label{eq3}
	^{\#  }	R_{\mathfrak{i}  }  \le
	\ ^{\#  }	\left \{ \mathfrak{l}_\pi:\pi \in \Sigma_{n}\ \right \}   = \mathscr{M}.
\end{equation}
Let
\begin{equation}\label{eq311}
A_{k} = \begin{cases}
		\left \{ [\mathfrak{i}]\in \mathbb{N}^n/\Sigma_n: | \mathfrak{i}  |=k \ and \ ^{\#  }	R_{\mathfrak{i}  } <\mathscr{M}    \right \},& k \geq 0, \\
\emptyset,& k < 0.
	\end{cases}
	\end{equation}
Then we have the following lemma.
\begin{lem}\label{lemmas2.2}
    For $k>0$ sufficiently large, if $l_{1},  l_{2},\cdots,l_{n}$ are not all equal, then there exists a constant $M>0$, such that $^{\#  } A_{k} \le Mk^{n-2}$.
\end{lem}
\begin{proof}
Let \begin{equation}\label{eq31001}A_{k}^{\prime}=\left \{  \mathfrak{i} \in \mathbb{N}^n: | \mathfrak{i}  |=k \ and \ ^{\#  }	R_{\mathfrak{i}  } <\mathscr{M}    \right \},~ k \geq 0.\end{equation}
Obviously,  $$^{\#  }A_{k} \leq ^{\#  }A_{k}^{\prime}.$$
	Set
	\begin{center}
		$B_{k}=\left \{  \mathfrak{i}  \in A_{k}^{\prime}:\exists \, 1\le j\le n \  and \  \exists \, \pi \in \Sigma_{n} \ such \  that \  i_{j}-l_{\pi(j)}<0  \   \right \},$
	\end{center}
	and
	\begin{center}
		$D_{k}=\left \{  \mathfrak{i} \in A_{k}^{\prime}:\forall \, 1\le j\le n \  and \  \forall \, \pi \in \Sigma_{n} \ such \  that \  i_{j}-l_{\pi(j)}\ge 0  \   \right \}.$
	\end{center}
	It is easy to see that \begin{equation}\label{eq4151}A_{k}^{\prime} = B_{k}\cup D_{k}.\end{equation}
To obtain the desired result, it suffices to show that there exist constants $M_1,M_2>0$ such that
	\begin{equation}\label{eq4}
		^{\#  } B_{k} \le M_{1}\left ( k+1 \right ) ^{n-2}
	\end{equation}
and 
	\begin{equation}\label{eq5}
		^{\#  } D_{k}  \le  M_{2} \left ( k+1 \right ) ^{n-2}.
	\end{equation}
For \eqref{eq4}, by the definition of $B_{k}$, we have
	\begin{center}
		$B_{k}=\bigcup\limits_{j=1}^{n} \left \{ \mathfrak{i}  \in A_{k}^{\prime}:\exists \, \pi \in \Sigma_{n} \ such \  that \  i_{j}-l_{\pi(j)}<0  \   \right \}.$
	\end{center}
For $ 1 \le j \le n,$	let
	
		$$C_{k,j}=\left \{ \mathfrak{i}  \in A_{k}^{\prime}:\exists \, \pi \in \Sigma_{n} \ such \  that \  i_{j}-l_{\pi(j)}<0  \   \right \},$$
then obviously
$$ \begin{aligned}
		C_{k,j} & \subseteq   \left \{  \mathfrak{i}\in\mathbb N^n:    |\mathfrak{i}|=k, \exists \, \pi \in \Sigma_{n} \ such \  that \ i_{j}-l_{\pi(j)}<0  \right \}.
	\end{aligned} $$
Set
	\begin{equation}\label{eq6}
		h=\max \left \{ {\left | l_{1} \right | ,\left | l_{2} \right | ,\cdots ,\left | l_{n} \right |  } \right \}.
	\end{equation}
Then, 
	$$\begin{aligned}^{\#}C_{k,j}
	\leq	^{\#  } \bigcup\limits_{i_{j}=0}^{h} \left \{ \mathfrak{i} \in \mathbb{N}^n: |\mathfrak{i}|=k  \   \right \}  \leq \left ( h+1 \right ) \left ( k+1 \right ) ^{n-2}.
	\end{aligned}$$
 Notice that $B_{k}=\bigcup\limits_{j=1}^{n} C_{k,j},$ hence $^{\#  } B_{k}  \le n\left ( h+1 \right ) \left ( k+1 \right ) ^{n-2}$.
Next we will prove (\ref{eq5}). 
	For $ \mathfrak{i} \in D_{k}$, by the definitions of $D_{k}$ and  $A_{k}^{\prime}$, we have
\begin{equation}\label{eq2.7.1}
 i_{j}-l_{\pi(j)}\ge 0,\quad \forall \, 1\le j\le n ,  \forall \, \pi \in \Sigma_{n} \text{ and }  ^{\#  } R_{\mathfrak{i}  }   <\mathscr{M}. 
 \end{equation}
  Then by \eqref{eq2} and \eqref{eq2.7.1},
\begin{equation}\label{eq2.7.}
^{\#  }\left\{ [\mathfrak{i}-\mathfrak{l}_\sigma]:\sigma \in \Sigma_n \right\}= ^{\#  } R_{\mathfrak{i}  } <\mathscr{M}.
\end{equation}
We claim that for any $\mathfrak{i}\in D_k$ there exist $\sigma_{0}, \tau_{0} \in \Sigma_{n}$,  satisfying $\mathfrak{l}_{\sigma_0}\not=\mathfrak{l}_{\tau_0}$, such that
	\begin{equation}\label{eq2.8.}
		[\mathfrak{i}-\mathfrak{l}_{\tau_0}] = [\mathfrak{i}-\mathfrak{l}_{\sigma_0}].
	\end{equation}
Otherwise, if  for all
$\sigma, \tau \in \Sigma_{n}$ such that $\mathfrak{l}_{\sigma}\not=\mathfrak{l}_{\tau}$, we have
	\begin{eqnarray*}
[\mathfrak{i}-\mathfrak{l}_{\tau}]\not= [\mathfrak{i}-\mathfrak{l}_{\sigma}].
\end{eqnarray*}
Then
$$^{\#  }\left\{ [\mathfrak{i}-\mathfrak{l}_\sigma]:\sigma \in \Sigma_n \right\}= \mathscr{M},
$$
 which contradicts to \eqref{eq2.7.} and the claim is proved.	For $\tau$, $  \sigma \in \Sigma_{n}$, $\tau\not=\sigma$, set
	\begin{eqnarray*}\label{eq2.7.11}
		D_{k,\tau ,\sigma}&=&\left \{   \mathfrak{i} \in D_{k}: \exists 1\leq a \neq b\leq n ~s.t.~ i_{a}-l_{\tau  (a)} =i_{b}-l_{\sigma  (b)}\right\}.
\end{eqnarray*}
Obviously,
\begin{eqnarray*}
 D_{k,\tau ,\sigma}\subseteq& \bigcup\limits_{1 \le a \ne b \le n} \{\mathfrak{i} \in \mathbb{N}^n: |\mathfrak{i}|=k, \ i_{a}-l_{\tau  (a)} =i_{b}-l_{\sigma  (b)} \}.
	\end{eqnarray*}
	Therefore by \eqref{eq2.8.},
	\begin{equation}\label{eq7}
		D_{k} \subseteq \bigcup\limits_{\tau \ne \sigma \in \Sigma_{n}}	D_{k,\tau ,\sigma} \subseteq \bigcup\limits_{\tau \ne \sigma \in \Sigma_{n}} \bigcup\limits_{1 \le a \ne b \le n}\left \{\mathfrak{i} \in \mathbb{N}^n: |\mathfrak{i}|=k, \ i_{a}-l_{\tau  (a)} =i_{b}-l_{\sigma  (b)}   \   \right \}.
	\end{equation}
	Notice that for $1 \le a \ne b \le n,$
 $$
 ^{\#  } \left \{\mathfrak{i}\in \mathbb{N}^n: |\mathfrak{i}|=k, \ i_{a}-l_{\tau  (a)} =i_{b}-l_{\sigma  (b)}   \   \right \}  \le \left ( k+1 \right ) ^{n-2}.
 $$ 
 Therefore by (\ref{eq7}), we have 
$$^{\#  } D_{k}  \leq (n!)^{2}\,^{\#  } D_{k,\tau ,\sigma} \le (n!)^{2}n^{2}\left ( k+1 \right ) ^{n-2}.$$
	Notice that $n$ is fixed, and the proof is completed.
\end{proof}


Let $\left \{ e _{i}  \right \} _{i=0}^{\infty } $ be an ONB for $\mathcal{H}$. It is well known that for $i_1<i_2<\cdots<i_n$,
$$
\left\|e^{k_{1}} _{i_1} \odot e_{i_2}^{k_{2}} \odot \cdots \odot e^{k_{n}}_{i_n}\right\|=\left(\frac{k_{1}!k_{2}!\cdots k_{n}!}{\left ( k_{1}+  k_{2}+\cdots+k_{n} \right )! }\right)^{1 / 2},
$$
where $$e^{k_{j}} _{i_j} = \underbrace{e _{i_j}\odot \cdots \odot e _{i_j} }_{k_{j}}, \ 1 \le j \le n.$$
For $m=0,1,2,\cdots ,\mathscr{M}$, set
\begin{equation}\label{eq8}
	E_{r,m} = \begin{cases}
	\left \{ [\mathfrak{i}]\in\mathbb N^n/\Sigma_n: \left | \mathfrak{i}  \right |=r \ and \   ^{\#  } R_{\mathfrak{i}  }  =m   \right \},& r \geq 0, \\
	\emptyset,& r < 0.
\end{cases}
\end{equation}
Obviously, $A_r=\bigcup\limits_{m=1}^{\mathscr{M}-1}E_{r,m}$.
For  simplicity, for $\mathfrak{i},\mathfrak{j}\in \mathbb N^n$, let
\begin{equation}\label{eq11}
		N_{\mathfrak{j}, \mathfrak{i}  }=\left \{ \pi\in\Sigma_{n}:[\mathfrak{i}+\mathfrak{l}_\pi]= [\mathfrak{j}] \right \}.
	\end{equation}
Then we have the following lemmas.
\begin{lem}\label{lem530}

	If $l_{1},l_{2},\cdots,l_{n}$  are not all equal, and for $\mathfrak{j}\in \mathbb N^n,$ there exist $1 \le s_0 < t_0 \le n$ such that $j_{s_0}=j_{t_0}.$ Then \begin{equation}\label{eq1001}^{\#  }\left \{[\mathfrak{j}-\mathfrak{l}_{\pi}]:\pi \in \Sigma_{n} \right \}
	\le \ ^{\#  }\left \{ \mathfrak{l}_{\pi} :\pi \in \Sigma_{n} \right \}  - 1.\end{equation}
\end{lem}
\begin{proof}
	
	Define $T \colon  \, \left \{ \mathfrak{l}_{\pi} :\pi \in \Sigma_{n} \right \}   \to \left \{[\mathfrak{j}-\mathfrak{l}_{\pi}] :\pi \in \Sigma_{n} \right \}$ by
\begin{align*}
T:	 \mathfrak{l}_{\pi}  \mapsto[\mathfrak{j}-\mathfrak{l}_{\pi}].
\end{align*}
Obviously, $T$ is well-defined and
surjective. Hence in order to show \eqref{eq1001}, it suffices to prove that $T$ is not injective. We will consider two cases.
\begin{case}$l_{s_0} \ne l_{t_0}.$

	Set $\pi_1=(1)$ and $\pi_2=\left ( s_0t_0 \right )$. Obviously,
	$\mathfrak{l}_{\pi_1}  \ne  \mathfrak{l}_{\pi_2}.$
	Since $j_{s_0}=j_{t_0}$, we have $j_{s_0}-l_{s_0}=j_{t_0}-l_{s_0}$ and $j_{t_0}-l_{t_0}=j_{s_0}-l_{t_0}$.
	Hence it is easy to see that $$ T\left( \mathfrak{l}_{\pi_1} \right) =[\mathfrak{j}-\mathfrak{l}_{\pi_1}]=[\mathfrak{j}-\mathfrak{l}_{\pi_2}]=T\left( \mathfrak{l}_{\pi_2}\right),$$ which implies that $T$ is not injective.
\end{case}
\begin{case}$l_{s_0}=l_{t_0}.$
	
	Since $l_{1},l_{2},\cdots,l_{n}$  are not all equal, there exists $1 \leq r_0 \leq n$, such that $l_{r_0} \neq l_{s_0}=l_{t_0}$.
Set $ \pi_1=(r_0s_0)$ and $ \pi_2= (r_0t_0)$.
	Obviously, $\mathfrak{l}_{\pi_1}  \ne  \mathfrak{l}_{\pi_2}.$ 
Since $j_{s_0}=j_{t_0}$ and $l_{s_0}=l_{t_0}$, we have $j_{s_0}-l_{r_0}=j_{t_0}-l_{r_0}$, $j_{t_0}-l_{t_0}=j_{s_0}-l_{s_0}$ and $j_{r_0}-l_{s_0}=j_{r_0}-l_{t_0}$.
Hence it is not difficult to see $$ T\left( \mathfrak{l}_{\pi_1} \right) =[\mathfrak{j}-\mathfrak{l}_{\pi_1}]=[\mathfrak{j}-\mathfrak{l}_{\pi_2}]=T\left( \mathfrak{l}_{\pi_2}\right),$$ which implies that $T$ is not injective. \qedhere
\end{case}
\end{proof}
\begin{lem}\label{lem175}
	Suppose that $l_{1},l_{2},\cdots,l_{n}$  are not all equal. Then for $k>0$ large enough, and every $\mathfrak{i},\mathfrak{j} \in \mathbb{N}^n,$ which satisfies that
$[\mathfrak{j}] \in E_{k+S_{\mathfrak{l}},\mathscr{M}}$ 
and $[\mathfrak{i}] \in R_{\mathfrak{j} },$
we have 
$$
\left \|  e_{[\mathfrak{i}]}  \right \|\leq \frac{^{\#  } N_{\mathfrak{j},\mathfrak{i} }\mathscr{M}}{n!\sqrt{n!}}.
$$
\end{lem}
\begin{proof}
	Since $[\mathfrak{j}]\in E_{k+S_{\mathfrak{l}},\mathscr{M}}$, by the definition of $E_{k+S_{\mathfrak{l}},\mathscr{M}}$ in (\ref{eq8}),
	\begin{equation}\label{eq9}
		^{\#  } R_{\mathfrak{j}  }  =\mathscr{M} .
	\end{equation}
We claim that for all $1 \le s < t \le n,$
	\begin{equation}\label{eq34}
		j_{s} \ne j_{t}.
	\end{equation}
	Otherwise, there exists $1 \le s_0 < t_0 \le n$, such that \begin{equation}\label{eq3411}j_{s_0}=j_{t_0}.\end{equation} By (\ref{eq2}), obviously
	$$R_{\mathfrak{j}  }    =
	\left \{ [\mathfrak{j}-\mathfrak{l}_{\pi}] :\pi \in \Sigma_{n} , j_{i}-l_{\pi(i)}  \ge 0, \forall \, 1\le i \le n\right \}.$$
	Notice that $l_{1},l_{2},\cdots,l_{n}$  are not all equal, hence by Lemma \ref{lem530},
		$$\begin{aligned}
		^{\#  } R_{\mathfrak{j}  }
		\le  \ ^{\#  }\left \{ [\mathfrak{j}-\mathfrak{l}_{\pi}] :\pi \in \Sigma_{n} \right \}
		  \le \ ^{\#  }\left \{ \mathfrak{l}_{\pi}  :\pi \in \Sigma_{n} \right \}  - 1
		 =  \ \mathscr{M}-1,
	\end{aligned}$$
which contradicts to \eqref{eq9}. The claim is proved.
	Write
	$$[i_1,\cdots,i_n]=[\underbrace{\widetilde{i_{1}},\cdots,\widetilde{i_{1}} }_{t_{1}} ,\cdots,\underbrace{\widetilde{i_{s}},\cdots,\widetilde{i_{s}} }_{t_{s}}], \mathrm{where} \ \widetilde{i_{1}} \ne \cdots \ne \widetilde{i_{s}}.$$
Since $[\mathfrak{i}] \in R_{\mathfrak{j} },$ we have $ N_{\mathfrak{j},\mathfrak{i} }  \neq \emptyset.$
	Therefore by (\ref{eq11}) and \eqref{eq34},
	$$ \begin{aligned}
		^{\#  } N_{\mathfrak{j},\mathfrak{i} }  &= \ ^{\#  } \left \{\pi \in \Sigma_{n}:[\mathfrak{i}+\mathfrak{l}_{\pi}]=[\mathfrak{j}] \right \}
		\ge \  t_{1}!t_{2}!\cdots t_{s}!n_{1}!\cdots n_{k}!.
	\end{aligned} $$
	Thus	
\begin{equation*}
	\left \|  e_{[\mathfrak{i}]}  \right \|=\sqrt{\frac{t_{1}!t_{2}!\cdots t_{s}!}{n!}}
	\leq \frac{t_{1}!t_{2}!\cdots t_{s}!}{\sqrt{n!}}
	\leq\frac{ ^{\#  } N_{\mathfrak{j},\mathfrak{i} }}{n_1!n_2!\cdots n_k!\sqrt{n!}}
	= \frac{^{\#  } N_{\mathfrak{j},\mathfrak{i} }\mathscr{M}}{n!\sqrt{n!}}.  \qedhere
\end{equation*}
\end{proof}
For $k\geq 0, n\geq 1,$ using the notation in \cite[Page 55]{FJbook}, denote $P(k,n)$  by the number of partitions of $k$ into at most $n$ parts, i.e.
\begin{equation}\label{pkn}P(k,n)=\ ^{\#  }\left \{ [k_{1},\cdots,k_{n}] :\sum\limits_{i=1}^{n}k_{i}=k, k_{i}\ge0  \right \}.\end{equation} For $k<0$, $n \ge 1$, let $P(k,n)=0$.
Moreover, for $k\geq 0$, $Q(k,n)$ is defined by the number of partitions of $k$ into at most $n$ ordered parts, i.e.
$$Q(k,n)=\ ^{\#  }\left \{ \left ( k_{1},\cdots,k_{n} \right ) :\sum\limits_{i=1}^{n}k_{i}=k, k_{i}\ge0  \right \}.$$
The following lemma may be well-known, but we state it here for the readers' convenience.
\begin{lem}\label{lemma2.5}
	For fixed $n \ge 1$,
	$$\lim_{k \to \infty} \frac{P(k,n)}{P(k+1,n)} =1.$$
\end{lem}
\begin{proof}
	For $n=1$,
	$$\lim_{k \to \infty} \frac{P(k,1)}{P(k+1,1)} =\lim_{k \to \infty} \frac{k}{k+1} =1.$$
	We will prove the lemma by induction.
	Suppose that
	\begin{equation}\label{eq13}
		\lim_{k \to \infty} \frac{P(k,m)}{P(k+1,m)} =1,
	\end{equation} and we will prove
	\begin{equation}\label{eq14}
		\lim_{k \to \infty} \frac{P(k,m+1)}{P(k+1,m+1)} =1.
	\end{equation}
	By \cite[Page 57]{FJbook}, for $\left | q \right |  <1$, we have
	$$\sum\limits_{k=0}^{\infty }P (k,m)q^{k} =\frac{1}{\left ( 1-q \right ) \left ( 1-q^{2} \right )\cdots\left ( 1-q^{m} \right )  }.$$
	It follows that
	$$\begin{aligned}
		\sum\limits_{k=0}^{\infty }P (k,m+1)q^{k} &=\left ( \sum\limits_{k=0}^{\infty }P (k,m)q^{k} \right ) \frac{1}{ 1-q^{m+1}   }
	\\	&=\left ( \sum\limits_{k=0}^{\infty }P (k,m)q^{k} \right )\sum\limits_{k=0}^{\infty }\left ( q^{m+1}  \right )^{k}
	\\&= \sum\limits_{k=0}^{\infty } \sum\limits_{i+j=k}P(i,m)f(j)q^{k},
	\end{aligned}$$
	where
\begin{equation}
f(j) =
\begin{cases}
1, & j = k(m + 1), \quad k = 0, 1, 2, \cdots, \\
0, & \text{otherwise}.
\end{cases}
\label{eq:131}
\end{equation}
	Hence
	\begin{equation}\label{eq15}
		P (k,m+1)=\sum\limits_{i+j=k}P(i,m)f(j).
	\end{equation}
     For $s \in \mathbb{N}$ and $0 \le t \le m,$ let $k=\left ( m+1 \right )s+t$. In order to prove (\ref{eq14}), it suffices to show that
     for every $0 \le t \le m,$
     $$    \lim_{s \to \infty}\frac{P\left ( \left ( m+1 \right ) s+t,m+1  \right ) }{P\left ( \left ( m+1 \right ) s+t+1,m+1  \right )}=1 .$$
     By (\ref{eq15}) and \eqref{eq:131}, we have
     $$\begin{aligned}
     	P\left ( \left ( m+1 \right ) s+t,m+1  \right ) &=\sum\limits_{i+j=\left ( m+1 \right ) s+t}P(i,m)f(j)
     	\\&= \sum\limits_{i=0}^{s } P(\left ( m+1 \right ) i+t,m)
     \end{aligned}$$
     and
          $$\begin{aligned}
     	P\left ( \left ( m+1 \right ) s+t+1,m+1  \right ) &=\sum\limits_{i+j=\left ( m+1 \right ) s+t+1}P(i,m)f(j)
     	\\&= \sum\limits_{i=-1}^{s } P(\left ( m+1 \right ) i+t+1,m).
     \end{aligned}$$
     Then by (\ref{eq13}) and Stolz theorem, we have
     \begin{align*}
     \lim_{s \to \infty}\frac{P\left ( \left ( m+1 \right ) s+t,m+1  \right ) }{P\left ( \left ( m+1 \right ) s+t+1,m+1  \right )} &=\lim_{s \to \infty}\frac{ \sum\limits_{i=0}^{s } P(\left ( m+1 \right ) i+t,m)}{\sum\limits_{i=-1}^{s } P(\left ( m+1 \right ) i+t+1,m)}
     \\&= \lim_{s \to \infty}\frac{  P(\left ( m+1 \right ) s+t,m)}{ P(\left ( m+1 \right ) s+t+1,m)}
     \\&= 1.   \qedhere
     \end{align*}
\end{proof}
For $1>\varepsilon >0$, let
\begin{equation}\label{eq161}\tilde{A}_{k+S_\mathfrak{l},\mathscr{M},\varepsilon } =\left \{ [\mathfrak{j}] \in E_{k+S_\mathfrak{l},\mathscr{M} }:\Gamma_{j_{m},-l_{\pi  (m)}} > 1 - \varepsilon,  \forall 1 \leq m \leq n, \forall \, \pi \in \Sigma_{n}   \right \}\end{equation}
and
$$\check{A}  _{k+S_\mathfrak{l},\mathscr{M},\varepsilon } =\left \{ [\mathfrak{j}] \in E_{k+S_\mathfrak{l},\mathscr{M} }:\exists \, \pi \in \Sigma_{n} \text{ and } \exists \, 1 \le m \le n ~s.t.~
\Gamma_{j_{m},-l_{\pi  (m)}} \le 1  - \varepsilon \right \}.$$
It is easy to see that
\begin{equation}\label{eq16}
	E_{k+S_\mathfrak{l},\mathscr{M} } = \tilde{A}_{k+S_\mathfrak{l},\mathscr{M},\varepsilon } \sqcup \check{A}  _{k+S_\mathfrak{l},\mathscr{M},\varepsilon },
\end{equation}
where $\sqcup$ represents the disjoint union.
 Then we have the following lemmas.
\begin{lem}\label{lemma2.6}
	If $\left \{ \alpha _{k}  \right \} _{k=0}^{\infty }$ satisfies the regularity condition and $\lim\limits_{k\to\infty}|\alpha_k| = 1$, then there exists a positive constant $C$, such that for $ k > 0$ large enough, 
	$$^{\#  } \check{A}_{k+S_\mathfrak{l},\mathscr{M} ,\varepsilon}  \le C k ^{n-2}.$$
\end{lem}
\begin{proof}	For $1 \le m \le n$, write
$$\check{B}  _{k+S_\mathfrak{l},\varepsilon,m }=\left \{ \mathfrak{j} \in \mathbb{N}^{n}: |\mathfrak{j}|=k+S_\mathfrak{l}, 
\exists \, \pi \in \Sigma_{n} ~s.t.~
\Gamma_{j_{m},-l_{\pi  (m)}} \le 1  - \varepsilon \right \}$$
and
$$\check{A}  _{k+S_\mathfrak{l},\mathscr{M},\varepsilon }^{\prime} =\left \{ \mathfrak{j} \in \mathbb{N}^{n}: [\mathfrak{j}] \in E_{k+S_\mathfrak{l},\mathscr{M} }, \exists \, \pi \in \Sigma_{n} \text{ and } \exists \, 1 \le m \le n ~s.t.~
\Gamma_{j_{m},-l_{\pi  (m)}} \le 1  - \varepsilon \right \}.$$
	It is easy to see that 
 \begin{equation}\label{eqprime}^\#\check{A}  _{k+S_\mathfrak{l},\mathscr{M},\varepsilon } \leq \ ^\#\check{A}  _{k+S_\mathfrak{l},\mathscr{M},\varepsilon }^{\prime} \end{equation} and
	\begin{eqnarray}\label{eq37}
			  &&\check{A}  _{k+S_\mathfrak{l},\mathscr{M},\varepsilon }^{\prime} \nonumber\\
			&\subseteq&   \left \{ \mathfrak{j} \in \mathbb{N}^{n}:|\mathfrak{j}|=k+S_\mathfrak{l}, 
			\exists \, \pi \in \Sigma_{n} \text{ and } \exists \, 1 \le m \le n ~s.t.~
			\Gamma_{j_{m},-l_{\pi  (m)}} \le 1  - \varepsilon \right \}\nonumber\\
            &=&\bigcup_{m=1}^{n} \check{B}  _{k+S_\mathfrak{l},\varepsilon,m }.
	\end{eqnarray}
Suppose that $^{\#  } \check{B}  _{k+S_\mathfrak{l},\varepsilon,m }\neq \emptyset.$ Then
for $\mathfrak{j} \in \check{B}  _{k+S_\mathfrak{l},\varepsilon,m }$, there exists a $\pi \in\Sigma_n,$ such that $l_{\pi  (m)}\neq 0,$ and $\Gamma_{j_{m},-l_{\pi  (m)}} \le 1 - \varepsilon$, i.e.
 \begin{equation}
1 - \varepsilon \ge \Gamma_{j_{m},-l_{\pi  (m)}} = \begin{cases}
	\min \left \{ \left | \alpha _{j_{m}} \right |, \cdots,\left | \alpha _{j_{m}-l_{\pi  (m)}-1} \right | \right \}, &l_{\pi  (m)}< 0, \\
	\min \left \{ \left | \alpha _{j_{m}-1} \right |, \cdots,\left | \alpha _{j_{m}-l_{\pi  (m)}} \right | \right \}, &l_{\pi  (m)} > 0.
\end{cases}
\label{eq:88}
\end{equation}
Since the weight $\alpha$ satisfies the regularity condition,   $| \alpha_{k} |\leq\lim\limits_{i\to\infty}|\alpha_i|=1 $ holds for any $k \in \mathbb{N}$. Thus, there exists an $N > 0$ such that for all $i > N$, \begin{equation}
	\left | \alpha_{i} \right | > 1 - \varepsilon.
	\label{eq:87}
\end{equation}	
Recall that $h=\max \left \{ {\left | l_{1} \right | ,\cdots ,\left | l_{n} \right |  } \right \}$ is introduced in (\ref{eq6}). Hence by \eqref{eq:88} and \eqref{eq:87}, for $\mathfrak{j} \in \check{B}  _{k+S_\mathfrak{l},\varepsilon,m }$,  $j_{m} \le N+h+1.$
It follows that for every $m\in \left \{ 1,2,\cdots,n \right \} $,
 \begin{equation}\label{eq38}
	\check{B}  _{k+S_\mathfrak{l},\varepsilon,m } \subseteq  \bigcup\limits_{j_{m}=0}^{N+h+1 } \left \{ \mathfrak{j} \in \mathbb{N}^{n}:|\mathfrak{j}|=k+S_\mathfrak{l}   \right \}.
\end{equation}
It is not hard to verify that for every fixed $0 \le t \le N+h+1$,
 \begin{equation}
	^\#\left \{ \mathfrak{j} \in \mathbb{N}^{n}:j_m=t, |\mathfrak{j}|=k+S_\mathfrak{l}   \right \}  \le \left ( k+S_\mathfrak{l}+1  \right ) ^{n-2},
\end{equation}
then \begin{eqnarray}\label{eq39}
^{\#  }\check{B}  _{k+S_\mathfrak{l},\varepsilon,m }  \le (N+h+2)\left ( k+S_\mathfrak{l}+1  \right ) ^{n-2}.\end{eqnarray}
Thus by \eqref{eqprime}, (\ref{eq37}) and (\ref{eq39}), there is a positive constant $C$ such that for $k>0$ large enough, 
\begin{equation*}
	^{\#  }	\check{A}  _{k+S_\mathfrak{l},\mathscr{M},\varepsilon }  \le n(N+h+2)\left ( k+S_\mathfrak{l}+1  \right ) ^{n-2}\leq Ck^{n-2}.  \qedhere
\end{equation*}
\end{proof}
\begin{lem}\label{lem2.7}
	For $0 < \varepsilon < 1$, we have
		$$ \lim\limits_{k\to \infty}\frac{^{\#  } \tilde{A}_{k+S_\mathfrak{l},\mathscr{M},\varepsilon } }{P(k,n) }= 1.$$
\end{lem}
\begin{proof}
Obviously,
$$
\bigcup_{j=0}^{\mathscr{M} }  E_{k+S_\mathfrak{l} ,j}  =\left \{ [\mathfrak{j}]\in \mathbb N^n/\Sigma_n : \left | \mathfrak{j}\right |=k+S_\mathfrak{l}   \right \}
$$
and
 $$
 ^{\#  } \left \{ [\mathfrak{j}]\in\mathbb N^n/\Sigma_n : \left | \mathfrak{j} \right |=k+S_\mathfrak{l}  \right \}  =P(k+S_\mathfrak{l}, n).
 $$
	Notice that for $j\in \left \{ 0,1,\cdots,\mathscr{M} \right \} $, $E_{k+S_\mathfrak{l}, j}$ are mutually disjoint. Then we have
	\begin{equation}\label{eq18}
		^{\#  }E_{k+S_\mathfrak{l},\mathscr{M}}   = P(k+S_\mathfrak{l},n) - \sum\limits_{j=0}^{\mathscr{M}-1} \  ^{\#  }E_{k+S_\mathfrak{l},j}     .
	\end{equation}
	It follows from (\ref{eq16}) that
	\begin{equation}\label{eq19}
		\begin{aligned}
			^{\#  } \tilde{A}_{k+S_\mathfrak{l},\mathscr{M},\varepsilon }
			  &= \ ^{\#  }E_{k+S_\mathfrak{l},\mathscr{M}}   - \ ^{\#  } \check{A}_{k+S_\mathfrak{l},\mathscr{M},\varepsilon }
			\\&=  P(k+S_\mathfrak{l},n) - \sum\limits_{j=0}^{\mathscr{M}-1}\  ^{\#  }E_{k+S_\mathfrak{l},j} - \ ^{\#  } \check{A}_{k+S_\mathfrak{l},\mathscr{M},\varepsilon }.
		\end{aligned}
	\end{equation}
	For every $0 \le j \le \mathscr{M}-1$ and sufficiently large $k$, by \eqref{eq8}, \eqref{eq311} and Lemma 2.2, there exists a constant $M>0$ such that
	$$^{\#  } E_{k+S_\mathfrak{l},j}   \le \ ^{\#  } A_{k+S_\mathfrak{l}} \le M\left ( k+S_\mathfrak{l}\right ) ^{n-2}.$$
	Additionally, it is not difficult to see that for sufficiently large $k$,
	\begin{equation}\label{eq191} P(k+S_\mathfrak{l},n) \ge \frac{Q(k+S_\mathfrak{l},n)}{n!} = \frac{C_{k+S_\mathfrak{l}+n-1}^{n-1}}{n!}\end{equation}
    and
    \begin{equation}\label{eq1911}\lim\limits_{k \to \infty} \frac{\frac{C_{k+S_\mathfrak{l}+n-1}^{n-1}}{n!}}{k^{n-1}} > 0.\end{equation}
	Hence for every $0 \le j \le \mathscr{M}-1$,
	$$\varlimsup_{k\to \infty} \frac{^{\#  } E_{k+S_{\mathfrak{l}},j}  }{P(k+S_{\mathfrak{l}},n)}=0.$$
	By Lemma \ref{lemma2.6}, \eqref{eq191} and \eqref{eq1911},
	$$\varlimsup_{k\to \infty} \frac{^{\#  } \check{A}_{k+S_{\mathfrak{l}},\mathscr{M},\varepsilon }  }{P(k+S_{\mathfrak{l}},n)} = 0.$$
	Therefore, by (\ref{eq19}),
	$$
\lim\limits_{k\to \infty}\frac{^{\#  } \tilde{A}_{k+S_{\mathfrak{l}},\mathscr{M},\varepsilon }  }{P(k+S_{\mathfrak{l}},n) }=\lim\limits_{k\to \infty}\frac{P(k+S_{\mathfrak{l}},n) - \sum\limits_{j=0}^{\mathscr{M}-1} \ ^{\#  } E_{k+S_{\mathfrak{l}},j}   - \ ^{\#  } \check{A}_{k+S_{\mathfrak{l}},\mathscr{M},\varepsilon }  }{P(k+S_{\mathfrak{l}},n) } = 1.
$$
	Thus, by Lemma \ref{lemma2.5}, we have
		\begin{equation*}
			\lim\limits_{k\to \infty}\frac{^{\#  } \tilde{A}_{k+S_{\mathfrak{l}},\mathscr{M},\varepsilon }  }{P(k,n) }= 1. \qedhere
		\end{equation*}
\end{proof}
\begin{lem}\label{lem248}
For $k \geq 0$ and $0 \leq i_1 \leq \cdots \leq i_n$ satisfying $|\mathfrak{i}|=k,$ we have that for $[\mathfrak{j}] \in \tilde{A}_{k+S_{\mathfrak{l}},\mathscr{M},\varepsilon }(0 < \varepsilon < 1)$ 
 and $\pi \in N_{\mathfrak{j}, \mathfrak{i}  },$
it holds
\begin{equation*}\label{eq25111} \left | \beta_{i_{a},l_{\pi  (a)}} \right | >\left ( 1 - \varepsilon \right ) ^{ \left | l_{\pi(a)} \right |   },\quad \forall 1 \le a \le n.\end{equation*}
\end{lem}
\begin{proof}
Since $\pi \in N_{\mathfrak{j}, \mathfrak{i}  },$ by \eqref{eq11},
\begin{equation*}\label{eq26}
[\mathfrak{i}+\mathfrak{l}_\pi]= [\mathfrak{j}].
\end{equation*}
Then for every fixed $1 \le a \le n$, there exists a $1 \le b \le n$, such that $i_{a}+l_{\pi  (a)}=j_{b}.$ Choose a $\tau_{0} \in \Sigma_{n}$ such that $\tau_{0}(b) = \pi (a)$. Hence from $[\mathfrak{j}] \in \tilde{A}_{k+S_{\mathfrak{l}},\mathscr{M},\varepsilon }$ and \eqref{eq161}, we have
\begin{equation}\label{eq3611}
\Gamma_{i_{a}+l_{\pi  (a)},-l_{\pi(a)}} =	\Gamma_{i_{a}+l_{\pi  (a)},-l_{\tau_0(b)}} = \Gamma_{j_{b},-l_{\tau_0(b)}} > 1 - \varepsilon.
\end{equation}
 Then by \eqref{eq71} and (\ref{eq3611}),
\begin{equation}\label{eq40}
1 - \varepsilon <\Gamma_{i_{a}+l_{\pi  (a)},-l_{\pi(a)}}= \begin{cases}
	\min \left \{ \left | \alpha _{i_{a}+l_{\pi(a)}} \right |, \cdots,\left | \alpha _{i_{a}-1} \right | \right \} ,&-l_{\pi(a)}> 0, \\
1 ,&-l_{\pi(a)}= 0, \\
	\min \left \{ \left | \alpha _{i_{a}+l_{\pi(a)}-1} \right |, \cdots,\left | \alpha _{i_{a}} \right | \right \},& -l_{\pi(a)} < 0.  
\end{cases}
\end{equation}
Hence by (\ref{eq1}) and (\ref{eq40}),
\begin{equation*}
\left ( 1 - \varepsilon \right ) ^{ |l_{\pi(a)}|  } <\left | \beta_{i_{a},l_{\pi  (a)}} \right |= \begin{cases}
	\left | \alpha _{i_{a}-1} \cdots \alpha _{i_{a}+l_{\pi(a)}} \right |,&-l_{\pi(a)}> 0, \\
1 ,&-l_{\pi(a)}= 0, \\
	\left | \alpha _{i_{a}} \cdots \alpha _{i_{a}+l_{\pi(a)}-1} \right |,& -l_{\pi(a)} < 0.  
\end{cases} 
\end{equation*}
The lemma is proved. \qedhere
\end{proof}

For any $\mathfrak{i} \in \mathbb{N}^n$
  such that $|\mathfrak{i}|= k$, let
 \begin{equation}\label{eq2233}\Sigma_{n,\mathfrak{i}}^{\prime}=\{\pi\in\Sigma_{n}: \mathfrak{i}+\mathfrak{l}_\pi \in \mathbb{N}^n\}\end{equation}
and \begin{equation}\label{eq22}
		P_{\mathfrak{i},k,t  }=\left \{ \pi\in\Sigma_{n}:[\mathfrak{i}+\mathfrak{l}_\pi]\in E_{k+S_{\mathfrak{l}},t } \right \}, \quad 0 \leq t \leq \mathscr{M}.
	\end{equation}
Obviously, 
$$ P_{\mathfrak{i},k,0 }=\emptyset, \quad \Sigma_{n,\mathfrak{i}}^{\prime}= \mathop{\bigsqcup}_{t=1}^{\mathscr{M}}P_{\mathfrak{i},k,t},$$
where $\bigsqcup$ denotes the disjoint union.
Moreover, by the definitions of $P_{\mathfrak{i},k,t  },$ $N_{\mathfrak{j}, \mathfrak{i}  },$ and $E_{k+S_{\mathfrak{l}},t }$  in \eqref{eq22}, \eqref{eq11} and \eqref{eq8}, it is easy to see that for $t\geq 1,$
\begin{eqnarray}\label{eq3.1}
P_{\mathfrak{i},k,t  }= \bigsqcup\limits_{[\mathfrak{j}]\in E_{k+S_{\mathfrak{l}},t}} N_{\mathfrak{j},\mathfrak{i}}.
\end{eqnarray}
Hence 
\begin{eqnarray}\label{eq3.1111}
\Sigma_{n,\mathfrak{i}}^{\prime} = \mathop{\bigsqcup}_{t=1}^{\mathscr{M}}\bigsqcup\limits_{[\mathfrak{j}] \in E_{k+S_{\mathfrak{l}},t}} N_{\mathfrak{j},\mathfrak{i}}.
\end{eqnarray}

{\section{Proof of  the main result}

In this section, we prove our main result Theorem \ref{thm1.1}, which will be divided in two subsections.  

\subsection{Proof of ``$(1)\Leftrightarrow (2)$"}

	For	$(1)\Rightarrow (2)$,
assume  $0\not=\mu = \lim\limits_{i \rightarrow \infty}\alpha_i $. Obviously, for every $(l_{1},l_{2},\cdots,l_{n})\in \mathbb{Z}^n,$
$$
\left\|S_{\frac{\alpha}{\mu}, l_1}\odot\cdots\odot S_{\frac{\alpha}{\mu},l_n}\right\|={1\over |\mu|^{|\mathfrak{l}|}}\left\| S_{\alpha,{l_{1}}}\odot S_{\alpha,{l_{2}}}\odot \cdots \odot S_{\alpha,{l_{n}}}\right\|.
$$
Without loss of generality, assume that $\lim\limits_{i\to\infty}\alpha_i=1$. We claim that for every $(l_{1},l_{2},\cdots,l_{n})\in \mathbb{Z}^n,$
		\begin{equation}\label{eq20}
				\left \|  S_{\alpha,{l_{1}}}\odot S_{\alpha,{l_{2}}}\odot \cdots \odot S_{\alpha,{l_{n}}}  \right \|  = 1.
		\end{equation}
		By \cite[Proposition 3.4]{GA},
		$$\left \|  S_{\alpha,{l_{1}}}\odot S_{\alpha,{l_{2}}}\odot \cdots \odot S_{\alpha,{l_{n}}}  \right \|  \le 1.$$
		The proof of the claim only requires proving
		\begin{equation}\label{eq20case}\left \|  S_{\alpha,{l_{1}}}\odot S_{\alpha,{l_{2}}}\odot \cdots \odot S_{\alpha,{l_{n}}}  \right \|  \ge 1.\end{equation}
We will consider this in two cases.
		\begin{case1}
		 $l_{1}=l_{2}=\cdots=l_{n}=l$ are all equal.
	
Without loss of generality assume $l>0$.   Since  $ 1=\lim\limits_{i \to \infty } \left | \alpha_{i}  \right | \ge \left | \alpha_{m}  \right | $ for any $m \in \mathbb{N}$, it follows that for any $1 > \varepsilon > 0$, there exists an $N > 0$ such that for all $i > N$, 
\begin{equation}\label{eq2011}1 > \left | \alpha_{i} \right |  > 1  - \varepsilon.\end{equation}	Choose $\left \{ i_{j} \right \} _{j=1}^{n} $ large enough such that
	\begin{equation}\label{eq21}
		N   < i_{1} < i_{2} < \cdots < i_{n}.
	\end{equation}
	Obviously,
	$$	\begin{aligned}
				\left ( S_{\alpha}^l\odot S_{\alpha}^l\odot \cdots \odot S_{\alpha}^l  \right ) \left ( e_{[\mathfrak{i}]}    \right )
		=\left(\prod\limits_{j=1}^n \prod\limits_{k=1}^l\alpha_{i_{j}+k-1}\right)\left (  e_{i_{1}+l }\odot e_{i_{2}+l }\odot \cdots \odot e_{i_{n} +l} \right ).	
	\end{aligned}$$
	Since for $j\not= k$, $i_j\not=i_k$, we have
	$$\|  e_{i_{1} }\odot e_{i_{2} }\odot \cdots \odot e_{i_{n}}\|=\left \|  e_{i_{1}+l }\odot e_{i_{2}+l }\odot \cdots \odot e_{i_{n} +l} \right \| ={ 1\over \sqrt{n!}},$$
	 hence by \eqref{eq2011},
	$$	
					\left \| S_{\alpha}^l\odot S_{\alpha}^l\odot \cdots \odot S_{\alpha}^l\right \|\geq
			\left ( 1 - \varepsilon \right ) ^{nl}.
$$
	From the arbitrariness of $\varepsilon$, we have
	$$\left \|  S_{\alpha,{l_{1}}}\odot S_{\alpha,{l_{2}}}\odot \cdots \odot S_{\alpha,{l_{n}}}  \right \|  \ge 1.$$
	\end{case1}
	\begin{case2}
		 $l_{1},l_{2},\cdots,l_{n}$  are not all equal.
		\begin{case2.1}\label{case2.1} $\alpha_i \geq 0$ for $i\geq 0.$

For $0 \leq i_1 \leq \cdots \leq i_n$ with $|\mathfrak{i}|=k$, let $a_{[\mathfrak{i}]}=\frac{1}{\sqrt{P(k,n)} }$. Obviously $\sum\limits\limits_{0 \leq i_1 \leq \cdots \leq i_n, |\mathfrak{i}|=k} a_{[\mathfrak{i}]} \frac{ e_{[\mathfrak{i}]} }{\left \|  e_{[\mathfrak{i}]}  \right \| }$ is a unit vector.
To calculate the norm of $S_{\alpha,{l_{1}}}\odot \cdots \odot S_{\alpha,{l_{n}}}$, we will estimate the norms of $S_{\alpha,{l_{1}}}\odot \cdots \odot S_{\alpha,{l_{n}}}\left(\sum\limits_{0 \leq i_1 \leq \cdots \leq i_n, |\mathfrak{i}|=k}a_{[\mathfrak{i}]} \frac{ e_{[\mathfrak{i}]} }{\left \|  e_{[\mathfrak{i}]}  \right \| }\right)$. 
By the definition of $\Sigma_{n,\mathfrak{i}}^{\prime}$ in \eqref{eq2233}, it is easy to see that
for $\pi \in \Sigma_{n}\backslash \Sigma_{n,\mathfrak{i}}^{\prime},$  $$\beta_{i_{1},l_{\pi  (1)}} \cdots\beta_{i_{n},l_{\pi  (n)}}=0.$$ Hence by Lemma \ref{lem2.1}, \eqref{eq16} and \eqref{eq3.1111}, for $k>0$ large enough, we have
	\begin{equation}\label{eq24}
			\begin{aligned}
			&	\left ( S_{\alpha,{l_{1}}}\odot \cdots \odot S_{\alpha,{l_{n}}}  \right )
			 \sum\limits_{0 \leq i_1 \leq \cdots \leq i_n, |\mathfrak{i}|=k}a_{[\mathfrak{i}]} \frac{ e_{[\mathfrak{i}]} }{\left \|  e_{[\mathfrak{i}]}  \right \| }
\\ =& \sum\limits_{0 \leq i_1 \leq \cdots \leq i_n, |\mathfrak{i}|=k}\frac{a_{[\mathfrak{i}]}}{n!}\sum\limits_{\pi  \in \Sigma_{n,\mathfrak{i}}^{\prime} }\frac{ \beta_{i_{1},l_{\pi  (1)}} \cdots\beta_{i_{n},l_{\pi  (n)}}e_{[\mathfrak{i}+\mathfrak{l}_\pi]}}{\left \|  e_{[\mathfrak{i}] } \right \| }
			\\=& \sum\limits_{[\mathfrak{j}]\in E_{k+S_{\mathfrak{l}},\mathscr{M} }  }
			\left (\sum\limits_{0 \leq i_1 \leq \cdots \leq i_n, |\mathfrak{i}|=k}
 \sum\limits_{\pi \in N_{\mathfrak{j}, \mathfrak{i}  } }\frac{a_{[\mathfrak{i}]}		\beta_{i_{1},l_{\pi  (1)}} \cdots\beta_{i_{n},l_{\pi  (n)}}}{n!\left \| e_{[\mathfrak{i}]} \right \|} \right )
			e_{[\mathfrak{j}]}
			\\  & ~ + \sum\limits_{t=1}^{\mathscr{M}-1}\sum\limits_{ [\mathfrak{j}]\in E_{k+S_{\mathfrak{l}},t }  }
			\left ( \sum\limits_{0 \leq i_1 \leq \cdots \leq i_n, |\mathfrak{i}|=k}
 \sum\limits_{\pi \in N_{\mathfrak{j}, \mathfrak{i}  } }\frac{ a_{[\mathfrak{i}]  }		\beta_{i_{1},l_{\pi  (1)}} \cdots\beta_{i_{n},l_{\pi  (n)}}}{n!\left \|e_{[\mathfrak{i}]} \right \|} \right )
			 	e_{[\mathfrak{j}]}  
			\\ =&  \sum\limits_{[\mathfrak{j}]\in \tilde{A}_{k+S_{\mathfrak{l}},\mathscr{M},\varepsilon }}
             \left ( \sum\limits_{0 \leq i_1 \leq \cdots \leq i_n, |\mathfrak{i}|=k}
 \sum\limits_{\pi \in N_{\mathfrak{j}, \mathfrak{i}  } }\frac{a_{[\mathfrak{i}]}			\beta_{i_{1},l_{\pi  (1)}}\cdots\beta_{i_{n},l_{\pi  (n)}}}{n!\left \| e_{[\mathfrak{i}]} \right \|} \right )
           e_{[\mathfrak{j}]}
           \\& ~ +\sum\limits_{[\mathfrak{j}]\in\check{A}  _{k+S_{\mathfrak{l}},\mathscr{M},\varepsilon }  }      \left (\sum\limits_{0 \leq i_1 \leq \cdots \leq i_n, |\mathfrak{i}|=k}
 \sum\limits_{\pi \in N_{\mathfrak{j}, \mathfrak{i}  } }\frac{a_{[\mathfrak{i}]}	 	\beta_{i_{1},l_{\pi  (1)}} \cdots\beta_{i_{n},l_{\pi  (n)}}}{n!\left \| e_{[\mathfrak{i}]} \right \|} \right )
          e_{[\mathfrak{j}]}
          \\  & ~ + \sum\limits_{t=1}^{\mathscr{M}-1}\sum\limits_{[\mathfrak{j}]\in E_{k+S_{\mathfrak{l}},t }  }
			\left (\sum\limits_{0 \leq i_1 \leq \cdots \leq i_n, |\mathfrak{i}|=k}
 \sum\limits_{\pi \in N_{\mathfrak{j}, \mathfrak{i}  } }\frac{a_{[\mathfrak{i}]}			\beta_{i_{1},l_{\pi  (1)}} \cdots\beta_{i_{n},l_{\pi  (n)}}}{n!\left \| e_{[\mathfrak{i}]} \right \|} \right )
			e_{[\mathfrak{j}]}.
		\end{aligned}
	\end{equation}
		For $\mathfrak{j} \in \mathbb{N}^{n}$ satisfying $[\mathfrak{j}] \in E_{k+S_{\mathfrak{l}},\mathscr{M} }  $, by (\ref{eq34}), we have
		$$j_{k_1} \ne j_{k_2}, \ 1 \le k_1 \ne k_2 \le n,$$ which implies that
		\begin{equation}\label{eq25}
			\left \| e_{[\mathfrak{j}]} \right \| =\frac{1}{\sqrt{n!} }.
		\end{equation}
Notice that
$$\sum\limits_{[\mathfrak{j}]\in \tilde{A}_{k+S_{\mathfrak{l}},\mathscr{M},\varepsilon }}
             \left ( \sum\limits_{0 \leq i_1 \leq \cdots \leq i_n, |\mathfrak{i}|=k}
 \sum\limits_{\pi \in N_{\mathfrak{j}, \mathfrak{i}  } }\frac{a_{[\mathfrak{i}]}			\beta_{i_{1},l_{\pi  (1)}}\cdots\beta_{i_{n},l_{\pi  (n)}}}{n!\left \| e_{[\mathfrak{i}]} \right \|} \right )
           e_{[\mathfrak{j}]},$$ $$\sum\limits_{[\mathfrak{j}]\in\check{A}  _{k+S_{\mathfrak{l}},\mathscr{M},\varepsilon }  }      \left (\sum\limits_{0 \leq i_1 \leq \cdots \leq i_n, |\mathfrak{i}|=k}
 \sum\limits_{\pi \in N_{\mathfrak{j}, \mathfrak{i}  } }\frac{a_{[\mathfrak{i}]}	 	\beta_{i_{1},l_{\pi  (1)}} \cdots\beta_{i_{n},l_{\pi  (n)}}}{n!\left \| e_{[\mathfrak{i}]} \right \|} \right )
          e_{[\mathfrak{j}]}$$
       and $$ \sum\limits_{t=1}^{\mathscr{M}-1}\sum\limits_{[\mathfrak{j}]\in E_{k+S_{\mathfrak{l}},t }  }
			\left (\sum\limits_{0 \leq i_1 \leq \cdots \leq i_n, |\mathfrak{i}|=k}
 \sum\limits_{\pi \in N_{\mathfrak{j}, \mathfrak{i}  } }\frac{a_{[\mathfrak{i}]}			\beta_{i_{1},l_{\pi  (1)}} \cdots\beta_{i_{n},l_{\pi  (n)}}}{n!\left \| e_{[\mathfrak{i}]} \right \|} \right )
			e_{[\mathfrak{j}]}$$
are pairwise orthogonal,
		hence by (\ref{eq24}) and (\ref{eq25}), we have
		\begin{equation}\label{eq2511}  \begin{aligned}
			&	\left \| \left ( S_{\alpha,{l_{1}}}\odot S_{\alpha,{l_{2}}}\odot \cdots \odot S_{\alpha,{l_{n}}}  \right )
			\sum\limits_{0 \leq i_1 \leq \cdots \leq i_n, |\mathfrak{i}|=k}a_{[\mathfrak{i}]} \frac{ e_{[\mathfrak{i}]} }{\left \|  e_{[\mathfrak{i}]}  \right \| }       \right \| ^{2}
			\\\ge& \left \| \sum\limits_{[\mathfrak{j}]\in \tilde{A}_{k+S_{\mathfrak{l}},\mathscr{M},\varepsilon }  }
			\left ( \sum\limits_{0 \leq i_1 \leq \cdots \leq i_n, |\mathfrak{i}|=k}
 \sum\limits_{\pi \in N_{\mathfrak{j}, \mathfrak{i}  } } \frac{a_{[\mathfrak{i}]}		\beta_{i_{1},l_{\pi  (1)}}\cdots\beta_{i_{n},l_{\pi  (n)}}}{n!\left \| e_{[\mathfrak{i}]} \right \|} \right )
				e_{[\mathfrak{j}]}   \right \| ^{2}
			\\= &\sum\limits_{[\mathfrak{j}]\in \tilde{A}_{k+S_{\mathfrak{l}},\mathscr{M},\varepsilon }} \left | \sum\limits_{0 \leq i_1 \leq \cdots \leq i_n, |\mathfrak{i}|=k}
 \sum\limits_{\pi \in N_{\mathfrak{j}, \mathfrak{i}  } }\frac{a_{[\mathfrak{i}]}		\beta_{i_{1},l_{\pi  (1)}} \cdots\beta_{i_{n},l_{\pi  (n)}} }{n!\left \| e_{[\mathfrak{i}]} \right \|} \right | ^{2}
			\left \| e_{[\mathfrak{j}]}  \right \| ^{2}
			\\= &\frac{1}{n!}\sum\limits_{[\mathfrak{j}]\in \tilde{A}_{k+S_{\mathfrak{l}},\mathscr{M},\varepsilon}}
			\left | \sum\limits_{0 \leq i_1 \leq \cdots \leq i_n, |\mathfrak{i}|=k}
 \sum\limits_{\pi \in N_{\mathfrak{j}, \mathfrak{i}  } }\frac{a_{[\mathfrak{i}]}		\beta_{i_{1},l_{\pi  (1)}} \cdots\beta_{i_{n},l_{\pi  (n)}} }{n!\left \| e_{[\mathfrak{i}]} \right \|} \right | ^{2}\\ \ge&  \frac{^{\#  } \tilde{A}_{k+S_{\mathfrak{l}},\mathscr{M},\varepsilon } }{P(k ,n) } \left ( 1 - \varepsilon \right ) ^{2\left ( \left | l_{1} \right |  + \cdots +\left | l_{n} \right | \right ) },
		\end{aligned}	\end{equation}
where the last inequality comes from Lemma \ref{lem175} and Lemma \ref{lem248}, together with the fact that $[\mathfrak{i}] \in R_{\mathfrak{j} }$ if and only if $N_{\mathfrak{j},\mathfrak{i} }  \neq \emptyset.$
		By Lemma \ref{lem2.7}
			$$ \lim\limits_{k\to \infty}\frac{^{\#  } \tilde{A}_{k+S_{\mathfrak{l}},\mathscr{M},\varepsilon }  }{P(k,n) }\left ( 1 - \varepsilon \right ) ^{2\left ( \left | l_{1} \right |  + \cdots +\left | l_{n} \right | \right ) } = \left ( 1 - \varepsilon \right ) ^{2\left ( \left | l_{1} \right | + \cdots +\left | l_{n} \right | \right ) },$$
		and hence $$\left \| S_{\alpha,{l_{1}}}\odot S_{\alpha,{l_{2}}}\odot \cdots \odot S_{\alpha,{l_{n}}}  \right \|  \ge \left ( 1 - \varepsilon \right ) ^{\left | l_{1} \right | + \cdots +\left | l_{n} \right |}.$$
		From the arbitrariness of $\varepsilon$, we have
		$$\left \|  S_{\alpha,{l_{1}}}\odot S_{\alpha,{l_{2}}}\odot \cdots \odot S_{\alpha,{l_{n}}}  \right \|  \ge 1.$$

\end{case2.1}
\begin{case2.2} General case: 
$\alpha_i \in \mathbb{C}$ for $i\geq 0$ and $\lim\limits_{i\to\infty}\alpha_i=1$.
Set  $$H_i=\overline{span\{e_{i+1},e_{i+2},\cdots\}}.$$
Obviously, $\lim\limits_{i\rightarrow \infty}\|(S_{\alpha}-S)|_{H_i}\|=0.$ 
Then it is easy to see that
\begin{equation}\label{eq2342343}\lim\limits_{i\rightarrow \infty}\left \|  (S_{\alpha,{l_{1}}}\odot S_{\alpha,{l_{2}}}\odot \cdots \odot S_{\alpha,{l_{n}}}-S_{{l_{1}}}\odot S_{{l_{2}}}\odot \cdots \odot S_{{l_{n}}})|_{H_i\odot \cdots \odot H_i}  \right \|=0.\end{equation}
Notice that \begin{equation}\label{eq23423438}\left \|  S_{\alpha,{l_{1}}}\odot S_{\alpha,{l_{2}}}\odot \cdots \odot S_{\alpha,{l_{n}}}\right\|\geq \left \|  S_{\alpha,{l_{1}}}\odot S_{\alpha,{l_{2}}}\odot \cdots \odot S_{\alpha,{l_{n}}}|_{H_i
\odot \cdots \odot H_i}\right\|.\end{equation} 
Hence in order to show that $$\left \|  S_{\alpha,{l_{1}}}\odot S_{\alpha,{l_{2}}}\odot \cdots \odot S_{\alpha,{l_{n}}}\right\|\geq 1,$$ by \eqref{eq2342343} and \eqref{eq23423438}, it suffices to prove that for $i\geq 0,$ \begin{equation}\label{eq2342345}\|S_{{l_{1}}}\odot S_{{l_{2}}}\odot \cdots \odot S_{{l_{n}}}|_{H_i\odot \cdots \odot H_i}\|\geq 1.\end{equation}
Set $$(1)_i=\left( \underbrace{0,\cdots,0 }_{i+1} ,1,1,1,\cdots \right), \quad i\geq 0.$$
By some direct calculations, 
\begin{equation}\label{eq234234}\|S_{{l_{1}}}\odot S_{{l_{2}}}\odot \cdots \odot S_{{l_{n}}}|_{H_i\odot \cdots \odot H_i}\|\geq \|S_{(1)_i,{l_{1}}}\odot S_{(1)_i,{l_{2}}}\odot \cdots \odot S_{(1)_i,{l_{n}}}\|, \quad i\geq0.\end{equation}
By Case 2.1, \begin{equation}\label{eq2342341}\|S_{(1)_i,{l_{1}}}\odot S_{(1)_i,{l_{2}}}\odot \cdots \odot S_{(1)_i,{l_{n}}}\|=1, \quad  i\geq0.\end{equation}
Therefore  \eqref{eq2342345} holds.
\end{case2.2}
	\end{case2}
Now, we prove	$(2)\Rightarrow(1)$.
Without loss of generality, we suppose $\sup\{|\alpha_i| : i \geq 0\}=1.$
If $\{\alpha_{i}\}_{i=0}^{\infty}$ does not satisfy the regularity condition in Definition \ref{def:regularity}, 
then there exists a finite subset $\{i_j\}_{j=1}^N$ and $\delta < 1$ such that
$$
\begin{cases} 
|\alpha_{i_j}|=1, & j=1,\cdots,N, \\ 
|\alpha_i|< \delta, & i\not\in\{i_j: j=1,\cdots,N \}.
\end{cases}$$
It follows from the definition of $\beta_{i, t}$ in \eqref{eq1} that there exists a constant $K\geq 2$ such that, for any non-negative integers $j_1$
  and $j_2$
  satisfying $0 \leq j_1\leq j_2$,
  we have $$|\beta_{j_1 - K, K} \beta_{j_2 - K, K}| < \delta.$$
 We claim that
\begin{equation}\label{eq2.34}
\left \|  S_{\alpha}^K\odot S_{\alpha}^K\right \| <  \left \|   S_{\alpha}^K\right \|\left \| S_{\alpha}^K\right \|=1.
\end{equation}
In fact, for $(i_1,i_2) \in \mathbb{Z}^2$, let $a_{[i_1, i_2] } \in \mathbb{C}$
such that $a_{[i_1, i_2]} = 0$ for $(i_1, i_2) \notin \mathbb{N}^2$. By Lemma \ref{lem2.1}, for every $ k \geq 0,$
	\begin{equation*}
				\left( S_{\alpha,{K}}\odot S_{\alpha,{K}}\right)
			\sum_{0 \leq i_1 \leq  i_2 \atop i_1 +  i_2 = k}a_{[i_1 , i_2]} {\left(  e_{i_{1} }\odot e_{i_{2} } \right) }
=\sum_{0 \leq i_1 \leq  i_2 \atop i_1 +  i_2 = k}a_{[i_1 , i_2]}\beta_{i_{1},K} \beta_{i_{2},K} \left (  e_{i_{1}+K }\odot e_{i_{2}+K } \right ).
\end{equation*}
Therefore 
\begin{equation}\label{eq2.36}
			\begin{aligned}
	&	\left\|\left( S_{\alpha,{K}}\odot S_{\alpha,{K}}\right)
			\sum_{0 \leq i_1 \leq  i_2 \atop i_1 +  i_2 = k}a_{[i_1 , i_2]}\frac{ e_{i_{1} }\odot e_{i_{2} } }{\left \|  e_{i_{1} }\odot e_{i_{2} } \right \| }\right\|^2
\\&\leq \sum\limits_{0 \leq j_1 \leq  j_2 \atop j_1 +  j_2 = k+2K   }
		  \left| \frac{a_{[j_1 -K,j_2-K] }	\beta_{j_1 -K,K} \beta_{j_2-K,K}}{\left \|  e_{j_1 -K }\odot e_{j_2-K}  \right \|}  \right|^2
		\left\|e_{j_{1} }\odot e_{j_{2}}  \right\|^2.
			\end{aligned}
\end{equation}
Thus when $\sum\limits_{0 \leq i_1 \leq  i_2 \atop i_1 +  i_2 = k}|a_{[i_1 , i_2]}|^2 =1,$ by \eqref{eq2.36},
\begin{equation}
			\begin{aligned}
			&	\left\|\left( S_{\alpha,{K}}\odot S_{\alpha,{K}}\right)
			\sum_{0 \leq i_1 \leq  i_2 \atop i_1 +  i_2 = k}a_{[i_1 , i_2]}\frac{ e_{i_{1} }\odot e_{i_{2} } }{\left \|  e_{i_{1} }\odot e_{i_{2} } \right \| }\right\|^2
\\&\leq \delta^2 \left(\sum\limits_{ 0 \leq j_1 \leq j_2 \atop j_1 +  j_2 = k+2K }
		 \left| \frac{a_{[j_1 -K,j_2-K] }	}{\left \|  e_{j_1 -K }\odot e_{j_2-K}  \right \|} \right|^2
		\left\|e_{j_{1} }\odot e_{j_{2}}  \right\|^2\right)
		 \\&\leq \delta^2  \left(\sum\limits_{ 0 \leq j_1 \le  j_2 \atop j_1 +  j_2 = k+2K }
		 \left| a_{[j_1 -K,j_2-K] } \right|^2 \right)
		 \leq \delta^2  \left(\sum\limits_{ 0 \leq j_1 \le  j_2 \atop j_1 +  j_2 = k }
		 \left| a_{\left\langle j_1 ,j_2\right\rangle } \right|^2 \right)
=\delta^2
<1.
\end{aligned}
\end{equation}
Notice that for $0\leq k_1 < k_2,$ 
$$\left\langle\left( S_{\alpha,{K}}\odot S_{\alpha,{K}}\right)
			\sum_{0 \leq i_1 \leq  i_2 \atop i_1 +  i_2 = k_1}a_{[i_1 , i_2]}\frac{ e_{i_{1} }\odot e_{i_{2} } }{\left \|  e_{i_{1} }\odot e_{i_{2} } \right \| }, \left( S_{\alpha,{K}}\odot S_{\alpha,{K}}\right)
			\sum_{0 \leq i_1 \leq  i_2 \atop i_1 +  i_2 = k_2}a_{[i_1 , i_2]}\frac{ e_{i_{1} }\odot e_{i_{2} } }{\left \|  e_{i_{1} }\odot e_{i_{2} } \right \| }\right\rangle=0, $$
therefore
\begin{equation*}
	\left\|\left( S_{\alpha,{K}}\odot S_{\alpha,{K}}\right)\right\|=\sup\limits_{k\in\mathbb N}\left\{	\left\|\left( S_{\alpha,{K}}\odot S_{\alpha,{K}}\right)
	\sum_{0 \leq i_1 \leq  i_2 \atop i_1 +  i_2 = k}a_{[i_1 , i_2]}\frac{ e_{i_{1} }\odot e_{i_{2} } }{\left \|  e_{i_{1} }\odot e_{i_{2} } \right \| }\right\|\right\}\leq\delta. 
	\tag*{\qed}
\end{equation*}

\subsection{Proof of ``$(1)\Leftrightarrow (3)$"}
In this subsection, we prove $(1)\Leftrightarrow (3)$. First, we need some notations and some lemmas. It will be seen that the lemmas listed below are similar to those in Section 2, however we need more techniques from antisymmetric tensor products to deal with the proofs. 
\begin{lem}\label{lemA2.1}
	For $\mathfrak{i}\in\mathbb N^n$,
	$$\begin{aligned}
		& \left ( S_{\alpha,{l_{1}}}\wedge S_{\alpha,{l_{2}}}\wedge \cdots \wedge S_{\alpha,{l_{n}}}  \right ) \left ( e_{i_{1} }\wedge e_{i_{2} }\wedge\cdots \wedge e_{i_{n} }    \right ) 
		\\ &=\frac{1}{n!} \sum\limits_{\pi  \in \Sigma _{n} }\beta_{i_{1},l_{\pi  (1)}} \beta_{i_{2},l_{\pi  (2)}}\cdots\beta_{i_{n},l_{\pi  (n)}} \left (  e_{i_{1}+l_{\pi  (1)} }\wedge e_{i_{2}+l_{\pi  (2)} }\wedge \cdots \wedge e_{i_{n} +l_{\pi  (n)}} \right ). 
	\end{aligned}$$
\end{lem}
In this section, we fix $d\geq 1.$ 
Let \begin{equation}\label{wdef}W=\{  \mathfrak{i} \in\mathbb N^n: d\leq i_1< \cdots < i_n\},\end{equation}
$$\mathfrak{W}=\{ \mathfrak{i}+\mathfrak{t} \in\mathbb N^n: \mathfrak{i} \in W, \mathfrak{t} \in\mathbb Z^n, |t_r|\leq |\mathfrak{l}|, i_1+t_{1}\neq\cdots \neq i_n+t_{n}\}$$
and
$$W^{\prime}=\{  \mathfrak{i} \in W: i_1\geq d+|\mathfrak{l}|, i_r- i_{r-1}> 4 |\mathfrak{l}| ~for~all~ 2\leq r\leq n\}.$$
Then we have the following lemma.
\begin{lem}\label{lemma3.2}
For $\mathfrak{j}^{\prime}\in W^{\prime}, \mathfrak{j} \in \mathfrak{W}\backslash W^{\prime},$ we have
 \begin{equation}\label{4.23}\left\langle e_{j_1^{\prime}} \wedge \cdots \wedge e_{j_n^{\prime}}, e_{j_1} \wedge \cdots \wedge e_{j_n}\right\rangle =0.\end{equation} 
\end{lem}
\begin{proof}
It is obvious that $W^\prime\subset W\subset\mathfrak{W}$ and for $j\in W\setminus W^\prime$, the orthogonality \eqref{4.23} is trivial. Next, assume $j\in \mathfrak{W}\setminus W$. Then \eqref{4.23} is also trivial if $j_{r_0}  <  d$ for some $1\leq r_0\leq n$. 

At last, assume that for all $1 \leq r\leq n,$ $j_{r}\geq d,$
i.e. there exists an integer $2\leq r_0\leq n$ such that $$j_{r_0-1}=i_{r_0-1}+t_{r_0-1}>j_{r_0}=i_{r_0}+t_{r_0}.$$
Hence $$i_{r_0}-i_{r_0-1} < t_{r_0-1}-t_{r_0}\leq 2|\mathfrak{l}|,$$
which implies that $$j_{r_0-1}-j_{r_0}=i_{r_0-1}+t_{r_0-1}-(i_{r_0}+t_{r_0})\leq |i_{r_0-1}-i_{r_0}|+|t_{r_0}-t_{r_0-1}|\leq 4|\mathfrak{l}|.$$
Notice that $j_t^{\prime}- j_{t-1}^{\prime}> 4 |\mathfrak{l}|$ for all $2\leq t\leq n.$ Hence 
\eqref{4.23} holds.
  \end{proof}
  
  For a set $A\subseteq \mathbb N^n,$ write \begin{equation}\label{wkn}A_{(k,n)}=^{\#  }\{\mathfrak{i}\in A: |\mathfrak{i}|=k\}.\end{equation}
Then we have the following lemma.
\begin{lem}
  For $k> 0$ sufficiently large, there exists a constant $C_1>0$, such that 
  \begin{equation}\label{eqA0311189}\mathfrak{W}_{(k,n)}- W^{\prime}_{(k,n)}\leq C_1k^{n-2}.\end{equation}
\end{lem}
\begin{proof}
For $\mathfrak{j}\in \mathfrak{W}\backslash W,$ by the proof of Lemma \ref{lemma3.2}, 
it can be shown easily that there exists a constant $C_1^{\prime}>0$, such that 
 \begin{equation}\label{eqA034689}\mathfrak{W}_{(k,n)}- W_{(k,n)}\leq C_1^{\prime}\left ( k+1 \right ) ^{n-2}.\end{equation}
 For $\mathfrak{j}\in W\backslash W^{\prime},$ by the definitions of $W$ and $W^{\prime},$ there exists an integer $2\leq r_1\leq n$ such that  $0< j_{r_1-1}-j_{r_1}\leq 4|\mathfrak{l}|$ or $j_1<d+|\mathfrak{l}|,$ which implies that there exists a constant $C_2^{\prime}>0$, such that 
 \begin{equation}\label{eqA0897689}W_{(k,n)}- W^{\prime}_{(k,n)}\leq C_2^{\prime}\left ( k+1 \right ) ^{n-2}.\end{equation}
Therefore by \eqref{eqA034689} and \eqref{eqA0897689}, \eqref{eqA0311189} holds.
\end{proof}
For $\mathfrak{i} \in\mathbb N^n,$ set
\begin{equation}\label{A78}
	\mathcal{R}_{\mathfrak{i}}=\left\{\mathfrak{j} \in\mathbb N^n: d\leq j_1< \cdots < j_n,  \exists \, \pi \in \Sigma_{n}~s.t.~  e_{j_1+l_{\pi(1)}} \wedge \cdots \wedge e_{j_n+l_{\pi(n)}}=e_{i_1} \wedge \cdots \wedge e_{i_n}\right\}.
\end{equation}
Then it is easy to see that
\begin{equation}\label{eq31211}
	^{\#  }	\mathcal{R}_{\mathfrak{i}  }  \le
	\ ^{\#  }	\left \{ \mathfrak{l}_\pi:\pi \in \Sigma_{n}\ \right \}   = \mathscr{M}.
\end{equation}
Set
 \begin{equation}\label{eqA311}
	\mathcal{A}_{k} = \begin{cases}
		\left \{ \mathfrak{i}\in \mathfrak{W}:  | \mathfrak{i}  |=k, \ ^{\#  }	\mathcal{R}_{\mathfrak{i}} <\mathscr{M}    \right \},& k \geq 0, \\
		\emptyset,& k < 0.
	\end{cases}
\end{equation}
Then we have the following lemma.
\begin{lem}\label{lemma33}
	For $k>0$ sufficiently large, if $l_{1},  l_{2},\cdots,l_{n}$ are not all equal, then there exists a constant $M>0$, such that $^{\#  }\mathcal{A}_{k} \le Mk^{n-2}$.
\end{lem}
\begin{proof} 
For $k>0,$ set 
\begin{equation}\label{eqA3111}
	\mathcal{A}_{k}^{\prime} =
		\left \{ \mathfrak{i} \in  W^{\prime}: | \mathfrak{i}  |=k \ and \ ^{\#  }	\mathcal{R}_{\mathfrak{i}  } <\mathscr{M}    \right \}.
\end{equation}
Obviously, 
$$\mathcal{A}_{k}\backslash \mathcal{A}_{k}^{\prime}=\left \{ \mathfrak{i} \in  \mathfrak{W} \backslash W^{\prime}: | \mathfrak{i}  |=k \ and \ ^{\#  }	\mathcal{R}_{\mathfrak{i}  } <\mathscr{M}    \right \}.$$
Hence by \eqref{eqA0311189}, there exists a constant $C_{1} > 0$ such that for $k> 0$ sufficiently large,
\begin{equation}\label{eqA03111}^{\#  }\mathcal{A}_{k}\leq ^{\#  }\mathcal{A}_{k}^{\prime}+ C_1k^{n-2}.\end{equation}
Recall that $$A_{k}^{\prime}=\left \{  \mathfrak{i} \in \mathbb{N}^n: | \mathfrak{i}  |=k \ and \ ^{\#  }	R_{\mathfrak{i}  } <\mathscr{M}    \right \},~ k \geq 0$$ in \eqref{eq31001}.
We claim that \begin{equation}\label{eqA003111}\mathcal{A}_{k}^{\prime} \subseteq  A_{k}^{\prime}.\end{equation}
In fact, for $\mathfrak{i} \in \mathcal{A}_{k}^{\prime},$ by the definition of $R_{\mathfrak{i}}$ in \eqref{eq2}, 
\begin{equation}\label{eqBURHUA}	R_{\mathfrak{i}  }\subseteq \{[\mathfrak{i}- \mathfrak{l}_{\pi}] \in \mathbb N^n/\Sigma_n : \pi \in \Sigma_n\}.\end{equation}
Since $\mathfrak{i} \in \mathcal{A}_{k}^{\prime}\subseteq W^{\prime}\subseteq W,$ we have $i_r- i_{r-1}> 4 |\mathfrak{l}|$ for all $2\leq r\leq n$ and $i_1\geq d+|\mathfrak{l}|$, hence
$d\leq i_1-l_{\pi(1)}<\cdots<i_n-l_{\pi(n)}$ for every $\pi \in \Sigma_n.$
Then by \eqref{A78}, it is easy to see that
\begin{equation}\label{eqRHUA}\{\mathfrak{i}-\mathfrak{l}_{\pi} \in \mathbb N^n: \pi \in \Sigma_n\} \subseteq \mathcal{R}_{\mathfrak{i}}.\end{equation}
Therefore by \eqref{eqBURHUA} and \eqref{eqRHUA}, we have $ ^{\#  }	R_{\mathfrak{i}  } \leq
 ^{\#  }	\mathcal{R}_{\mathfrak{i}  },$ which implies that \eqref{eqA003111} holds.
The claim is proved.
By \eqref{eq4151}, \eqref{eq4} and \eqref{eq5}, there exists a constant $C_2>0$, such that $^{\#  }A_{k}^{\prime}\leq C_2\left ( k+1 \right ) ^{n-2}.$ Hence 
\eqref{eqA03111} and \eqref{eqA003111} ensure that for $k> 0$ sufficiently large, there exists a constant $M>0$, such that $^{\#  }\mathcal{A}_{k} \le Mk^{n-2}$.
\end{proof}
For $m=0,1,2,\cdots ,\mathscr{M}$, set
\begin{equation}\label{eqA3}
	\mathcal{E}_{r,m} = \begin{cases}
		\left \{ \mathfrak{i}\in \mathfrak{W} :\left | \mathfrak{i}  \right |=r \ and \   ^{\#  } \mathcal{R}_{\mathfrak{i}  }  =m   \right \},& r \geq 0, \\
		\emptyset,& r < 0.
	\end{cases}
\end{equation}
For  simplicity, for $\mathfrak{i},\mathfrak{j}\in \mathbb N^n$, let
\begin{equation}\label{eqA11}
	\mathcal{N}_{\mathfrak{j}, \mathfrak{i}  }=\left \{ \pi\in\Sigma_{n}:e_{i_1+l_{\pi(1)}} \wedge \cdots \wedge e_{i_n+l_{\pi(n)}}=e_{j_1} \wedge \cdots \wedge e_{j_n} \right \}.
\end{equation}
		Then for $k> 0$ sufficiently large, $\mathfrak{j}  \in \mathcal{E}_{k+S_\mathfrak{l},\mathscr{M} } $ and	$\mathfrak{i} \in \mathcal{R}_{\mathfrak{j} },$
    we have $\mathcal{N}_{\mathfrak{j},\mathfrak{i} }  \neq \emptyset.$
	Therefore by (\ref{eqA11}),
	$$ \begin{aligned}
		^{\#  } \mathcal{N}_{\mathfrak{j},\mathfrak{i} }  &= \ ^{\#  } \left \{\pi \in \Sigma_{n}:e_{i_1+l_{\pi(1)}} \wedge \cdots \wedge e_{i_n+l_{\pi(n)}}=e_{j_1} \wedge \cdots \wedge e_{j_n} \right \}
		\ge \  n_{1}!\cdots n_{k}!.
	\end{aligned} $$
	Thus	
	\begin{eqnarray}\label{eqA34}
		\left \| e_{i_1} \wedge \cdots \wedge e_{i_n}  \right \|&=&\sqrt{\frac{1}{n!}}
		\leq\frac{ ^{\#  } \mathcal{N}_{\mathfrak{j},\mathfrak{i} }}{n_1!n_2!\cdots n_k!\sqrt{n!}}
		= \frac{^{\#  } \mathcal{N}_{\mathfrak{j},\mathfrak{i} }\mathscr{M}}{n!\sqrt{n!}}.
	\end{eqnarray}

For $0< \varepsilon < 1$, set
\begin{equation}
	\tilde{\mathcal{A}}_{k+S_\mathfrak{l},\mathscr{M},\varepsilon } =\left \{ \mathfrak{j}  \in \mathcal{E}_{k+S_\mathfrak{l},\mathscr{M} }:\Gamma_{j_{m},-l_{\pi  (m)}} > 1 - \varepsilon,  \forall 1 \leq m \leq n, \forall \, \pi \in \Sigma_{n}   \right \}
\end{equation}
and
\begin{equation}\label{xianzai4}\check{\mathcal{A}}  _{k+S_\mathfrak{l},\mathscr{M},\varepsilon } =\left \{ \mathfrak{j}  \in \mathcal{E}_{k+S_\mathfrak{l},\mathscr{M} }:\exists \, \pi \in \Sigma_{n} \text{ and } \exists \, 1 \le m \le n ~s.t.~
\Gamma_{j_{m},-l_{\pi  (m)}} \le 1  - \varepsilon \right \}.\end{equation}
It is easy to see that
\begin{equation}\label{eqA16}
	\mathcal{E}_{k+S_\mathfrak{l},\mathscr{M} } = \tilde{\mathcal{A}}_{k+S_\mathfrak{l},\mathscr{M},\varepsilon } \sqcup \check{\mathcal{A}}  _{k+S_\mathfrak{l},\mathscr{M},\varepsilon }.
\end{equation}
The following lemma can be derived from \eqref{eq37} and \eqref{eq39}, along with the fact that
$$\check{\mathcal{A}}_{k+S_\mathfrak{l},\mathscr{M} ,\varepsilon}\subseteq \left \{ \mathfrak{j} \in \mathbb{N}^{n}:|\mathfrak{j}|=k+S_\mathfrak{l}, 
			\exists \, \pi \in \Sigma_{n} \text{ and } \exists \, 1 \le m \le n ~s.t.~
			\Gamma_{j_{m},-l_{\pi  (m)}} \le 1  - \varepsilon \right\}.$$
\begin{lem}\label{lemmaA.6}
	If $\left \{ \alpha _{k}  \right \} _{k=0}^{\infty }$ satisfies the regularity condition and $\lim\limits_{k\to\infty}|\alpha_k| = 1$, then there exists a positive constant $C$, such that for sufficiently large $ k >0$, 
	\begin{equation}\label{xianzai3}^{\#  } \check{\mathcal{A}}_{k+S_\mathfrak{l},\mathscr{M} ,\varepsilon}  \le C k ^{n-2}.\end{equation}
\end{lem}
For $m=0,1,2,\cdots ,\mathscr{M}$, set
\begin{equation}\label{eqA3511}
	\mathcal{E}_{r,m}^{\prime} = \begin{cases}
		\left \{ \mathfrak{i} \in W^{\prime}, \left | \mathfrak{i}  \right |=r \ and \   ^{\#  } \mathcal{R}_{\mathfrak{i}  }  =m   \right \},& r \geq 0, \\
		\emptyset,& r < 0,
	\end{cases}
\end{equation}
and for $0< \varepsilon < 1$, set
\begin{equation}\label{eq316}
\tilde{\mathcal{A}}_{k+S_\mathfrak{l},\mathscr{M},\varepsilon }^{\prime} =
\tilde{\mathcal{A}}_{k+S_\mathfrak{l},\mathscr{M},\varepsilon } \cap \mathcal{E}_{k+S_\mathfrak{l},\mathscr{M} }^{\prime}
\end{equation}
and $$\check{\mathcal{A}}^{\prime}_{k+S_\mathfrak{l},\mathscr{M},\varepsilon }=\check{\mathcal{A}}_{k+S_\mathfrak{l},\mathscr{M},\varepsilon } \cap \mathcal{E}_{k+S_\mathfrak{l},\mathscr{M} }^{\prime}.$$
Then we have following lemma.
\begin{lem}\label{lemA2.7}
For $0 < \varepsilon < 1$, we have
	$$ \lim\limits_{k\to \infty}\frac{^{\#  } \tilde{\mathcal{A}}_{k+S_\mathfrak{l},\mathscr{M},\varepsilon }^{\prime} }{W_{(k,n)} }= 1.$$
\end{lem}
\begin{proof}
 For $k > 0$ sufficiently large,
 it is evident that
\begin{equation}\label{eq11181}
\bigcup_{j=0}^{\mathscr{M} }  \mathcal{E}_{k+S_\mathfrak{l} ,j}  =\left \{ \mathfrak{j} \in \mathfrak{W} : \left | \mathfrak{j}\right |=k+S_\mathfrak{l}   \right \}.
\end{equation}
	Notice that for $j\in \left \{ 0,1,\cdots,\mathscr{M} \right \} $, $\mathcal{E}_{k+S_\mathfrak{l}, j}$ are mutually disjoint. Then by (\ref{eqA16}) and \eqref{eq11181}, we have
	\begin{equation}\label{eq319}
		\begin{aligned}
			^{\#  } \tilde{\mathcal{A}}_{k+S_\mathfrak{l},\mathscr{M},\varepsilon }
			  &= \ ^{\#  }\mathcal{E}_{k+S_\mathfrak{l},\mathscr{M}}   - \ ^{\#  } \check{\mathcal{A}}_{k+S_\mathfrak{l},\mathscr{M},\varepsilon }
			\\&=   \mathfrak{W}_{(k+S_\mathfrak{l},n)}- \sum\limits_{j=0}^{\mathscr{M}-1}\  ^{\#  }\mathcal{E}_{k+S_\mathfrak{l},j} - \ ^{\#  } \check{\mathcal{A}}_{k+S_\mathfrak{l},\mathscr{M},\varepsilon }.
		\end{aligned}
	\end{equation}
	We claim that
	\begin{equation} \label{xianzai1}\lim\limits_{k\to \infty} \frac{^{\#  } \tilde{\mathcal{A}}_{k+S_{\mathfrak{l}},\mathscr{M},\varepsilon }  }{W_{(k+S_{\mathfrak{l}},n)}}=1.\end{equation} 
	In fact, for
	every $0 \le j \le \mathscr{M}-1$, by \eqref{eqA3}, \eqref{eqA311} and Lemma \ref{lemma33}, \begin{equation} \label{eqwkn111}^{\#  } \mathcal{E}_{k+S_\mathfrak{l},j}   \le \ ^{\#  } \mathcal{A}_{k+S_\mathfrak{l}} \le M\left ( k+S_\mathfrak{l}\right ) ^{n-2}.\end{equation}
By \eqref{pkn}, \eqref{wdef} and \eqref{wkn}, it is easy to see that there exists a positive constant $C(n,d)$ depending only on $n$ and $d$ such that 
\begin{equation} \label{eqwkn1}
	|P(k,n)-W_{(k,n)}|\leq C(n,d) k^{n-2}.
\end{equation}
     Hence from \eqref{eq191}, \eqref{eq1911}, \eqref{eqwkn1}, \eqref{eqwkn111} and Lemma \ref{lemmaA.6}, for every $0 \le j \le \mathscr{M}-1$,
	\begin{equation} \label{111eqwkn1}\varlimsup_{k\to \infty} \frac{^{\#  } \mathcal{E}_{k+S_{\mathfrak{l}},j}  }{W_{(k+S_{\mathfrak{l}},n)}}=0, \quad
	\varlimsup_{k\to \infty} \frac{^{\#  } \check{\mathcal{A}}_{k+S_{\mathfrak{l}},\mathscr{M},\varepsilon }  }{W_{(k+S_{\mathfrak{l}},n)}} = 0.\end{equation}
Additionally, \eqref{eq191}, \eqref{eq1911}, \eqref{eqA034689} and \eqref{eqwkn1} give 
\begin{equation} \label{dengjia}\lim\limits_{k\to \infty} \frac{\mathfrak{W}_{(k+S_{\mathfrak{l}},n)}  }{W_{(k+S_{\mathfrak{l}},n)}} = 1.\end{equation}
Hence \eqref{eq319}, \eqref{111eqwkn1} and \eqref{dengjia}  yield \eqref{xianzai1}.
The claim is proved.

 Notice that $$\tilde{\mathcal{A}}_{k+S_\mathfrak{l},\mathscr{M},\varepsilon } \backslash \tilde{\mathcal{A}}_{k+S_\mathfrak{l},\mathscr{M},\varepsilon }^{\prime}=\left \{ \mathfrak{j}  \in \mathcal{E}_{k+S_\mathfrak{l},\mathscr{M} }\backslash \mathcal{E}_{k+S_\mathfrak{l},\mathscr{M} }^{\prime}:\Gamma_{j_{m},-l_{\pi  (m)}} > 1 - \varepsilon,  \forall 1 \leq m \leq n, \forall \, \pi \in \Sigma_{n}   \right \} $$
 and by \eqref{eqA0311189}, there exists a positive constant $C_1$ such that
 $$\left|^{\# } \mathcal{E}_{k+S_\mathfrak{l},\mathscr{M} }-^{\# } \mathcal{E}_{k+S_\mathfrak{l},\mathscr{M}}^{\prime}\right|\leq
 \mathfrak{W}_{(k,n)}-W^{\prime}_{(k,n)}\leq
 C_1k^{n-2}.$$
Consequently, we have
\begin{equation} \label{xianzai2}\left|^{\# } \tilde{\mathcal{A}}_{k+S_\mathfrak{l},\mathscr{M},\varepsilon }-^{\# } \tilde{\mathcal{A}}_{k+S_\mathfrak{l},\mathscr{M},\varepsilon }^{\prime}\right|\leq C_1k^{n-2}.\end{equation}
	Therefore, by \eqref{eq191}, \eqref{eq1911}, \eqref{eqwkn1}, \eqref{xianzai1} and \eqref{xianzai2},
	$$
\lim\limits_{k\to \infty}\frac{^{\#  }  \tilde{\mathcal{A} }_{k+S_{\mathfrak{l}},\mathscr{M},\varepsilon } ^{\prime} }{W_{_{(k+S_{\mathfrak{l}},n)}} }= 1.
$$
By \eqref{eqwkn1} and Lemma \ref{lemma2.5}, $$ 
\lim\limits_{k\to \infty}\frac {W_{_{(k+S_{\mathfrak{l}},n)}} }{W_{(k,n)} }= 1.
$$
	Thus,
		\begin{equation*}
			\lim\limits_{k\to \infty}\frac{^{\#  }  \tilde{\mathcal{A} }_{k+S_{\mathfrak{l}},\mathscr{M},\varepsilon }^{\prime}   }{W_{(k,n)} }= 1. \qedhere
		\end{equation*}
\end{proof}
The proof of the following lemma is similar to that of Lemma \ref{lem248}; for completeness, we list it here without providing the proof.
\begin{lem}\label{lemA248}
	For $k \geq 0$ and $d \leq i_1 < \cdots < i_n$ satisfying $|\mathfrak{i}|=k,$ we have that for $\mathfrak{j} \in \tilde{\mathcal{A}}_{k+S_{\mathfrak{l}},\mathscr{M},\varepsilon }(0 < \varepsilon < 1)$ 
	and $\pi \in \mathcal{N}_{\mathfrak{j}, \mathfrak{i}  },$
	it holds
	\begin{equation*}\label{eq25111} \left | \beta_{i_{a},l_{\pi  (a)}} \right | >\left ( 1 - \varepsilon \right ) ^{ \left | l_{\pi(a)} \right |   },\quad \forall 1 \le a \le n.\end{equation*}
\end{lem}
For any $\mathfrak{i} \in \mathbb{N}^n$
	such that $|\mathfrak{i}|= k$ and $d\leq i_1 < \cdots < i_n$, let
	\begin{equation}\label{eqA2233}
		\Sigma_{n,\mathfrak{i}}^{\prime\prime}=\{\pi\in\Sigma_{n}: \mathfrak{i}+\mathfrak{l}_\pi \in \mathbb{N}^n, i_1+l_{\pi(1)} \neq \cdots \neq i_n+l_{\pi(n)} \}
	\end{equation}
	and \begin{equation}\label{eqA22}
		\mathcal{P}_{\mathfrak{i},k,t  }=\left \{ \pi\in\Sigma_{n}: \mathfrak{i}+\mathfrak{l}_\pi \in \mathcal{E}_{k+S_{\mathfrak{l}},t } \right \}, \quad 0 \leq t \leq \mathscr{M}.
	\end{equation}
	Obviously, 
	\begin{equation}\label{eqA2211} \mathcal{P}_{\mathfrak{i},k,0 }=\emptyset, \quad \Sigma_{n,\mathfrak{i}}^{\prime \prime}= \mathop{\bigsqcup}_{t=1}^{\mathscr{M}}\mathcal{P}_{\mathfrak{i},k,t}.\end{equation}
	Moreover, by the definitions of $\mathcal{P}_{\mathfrak{i},k,t  },$ $\mathcal{N}_{\mathfrak{j}, \mathfrak{i}  }$ and $\mathcal{E}_{k+S_{\mathfrak{l}},t }$  in \eqref{eqA22}, \eqref{eqA11} and \eqref{eqA3}, as well as Lemma \ref{lemma3.2},
	it is easy to see that for $t\geq 1,$
	\begin{eqnarray}\label{eq3.1}
		\mathcal{P}_{\mathfrak{i},k,t  }= \mathop{\bigcup}\limits_{\mathfrak{j}\in \mathcal{E}_{k+S_{\mathfrak{l}},t}} \mathcal{N}_{\mathfrak{j},\mathfrak{i}}=\left(\mathop{\bigcup}\limits_{\mathfrak{j}\in \mathcal{E}_{k+S_{\mathfrak{l}},t}\backslash
\mathcal{E}_{k+S_{\mathfrak{l}},t}^{\prime}} \mathcal{N}_{\mathfrak{j},\mathfrak{i}}\right)\bigsqcup \left(\bigsqcup\limits_{\mathfrak{j}\in \mathcal{E}_{k+S_{\mathfrak{l}},t}^{\prime}} \mathcal{N}_{\mathfrak{j},\mathfrak{i}}\right) .
	\end{eqnarray}
Write \begin{equation}\label{eqA221221}\Sigma_{n,\mathfrak{i}}^{\prime \prime\prime}=\left(\mathop{\bigsqcup}_{t=1}^{\mathscr{M}}\mathop{\bigcup}\limits_{\mathfrak{j}\in \mathcal{E}_{k+S_{\mathfrak{l}},t}\backslash
\mathcal{E}_{k+S_{\mathfrak{l}},t}^{\prime}} \mathcal{N}_{\mathfrak{j},\mathfrak{i}}\right).\end{equation}

Next, we are ready to  prove $(1)\Leftrightarrow (3)$ of Theorem \ref{thm1.1}.\\
$The~ proof ~of~ (1)\Leftrightarrow (3) ~of ~Theorem~ \ref{thm1.1}.$
	The proof is similar to that of $(1)\Leftrightarrow (2),$ and the only difference is the proof of Case 2 in $(1)\Rightarrow(2)$.
Assume $l_1,\cdots,l_n$ are not all equal, and let $\alpha_i \geq 0$ for $i\geq 0$ with $\lim\limits_{i\rightarrow \infty}\alpha_i=1.$ 
Similar to \eqref{eq2342343}, \eqref{eq23423438} and \eqref{eq2342345},
we need to prove for every fixed $d\geq 1,$ 
\begin{equation}\label{eqA20case}\left \|  S_{{\alpha,l_{1}}}\wedge S_{{\alpha},l_{2}}\wedge \cdots \wedge S_{{\alpha,l_{n}}}|_{H_{d-1}\odot \cdots \odot H_{d-1}}  \right \|  \ge 1.\end{equation}
			
		For $d \leq i_1 < \cdots < i_n$ with $|\mathfrak{i}|=k$, let $a_{\mathfrak{i}}=\frac{1}{\sqrt{W_{(k,n)}} }$. Obviously $\sum\limits\limits_{d \leq i_1 < \cdots < i_n, |\mathfrak{i}|=k} a_{\mathfrak{i}} \frac{ e_{i_1} \wedge \cdots \wedge e_{i_n} }{\left \|  e_{i_1} \wedge \cdots \wedge e_{i_n}  \right \| }$ is a unit vector.
		To calculate the norm of $S_{{\alpha,l_{1}}}\wedge S_{{\alpha,l_{2}}}\wedge \cdots \wedge S_{{\alpha,l_{n}}}|_{H_{d-1}\odot \cdots \odot H_{d-1}}$, we will estimate the norms of $$S_{{\alpha,l_{1}}}\wedge S_{{\alpha,l_{2}}}\wedge \cdots \wedge S_{{\alpha,l_{n}}}\left(\sum\limits_{d \leq i_1 < \cdots < i_n, |\mathfrak{i}|=k}a_{\mathfrak{i}} \frac{ e_{i_1} \wedge \cdots \wedge e_{i_n} }{\left \|  e_{i_1} \wedge \cdots \wedge e_{i_n}  \right \| }\right).$$ 
		By the definition of $\Sigma_{n,\mathfrak{i}}^{\prime \prime}$ in \eqref{eqA2233}, it is easy to see that
			for $\pi \in \Sigma_{n}\backslash \Sigma_{n,\mathfrak{i}}^{\prime \prime},$ 
 $$ e_{i_1+l_{\pi(1)}} \wedge \cdots \wedge e_{i_n+l_{\pi(n)}}=0.$$ Hence by Lemma \ref{lemA2.1},  \eqref{eqA2211}, \eqref{eq3.1}, \eqref{eqA221221} and \eqref{eqA16}, for $k>0$ large enough, we have
\begin{equation}\label{eq324}
			\begin{aligned}
				&	\left (S_{{\alpha,l_{1}}}\wedge S_{{\alpha,l_{2}}}\wedge \cdots \wedge S_{{\alpha,l_{n}}} \right )
				\sum\limits_{d \leq i_1 < \cdots < i_n, |\mathfrak{i}|=k}a_{\mathfrak{i}} \frac{ e_{i_1} \wedge \cdots \wedge e_{i_n} }{\left \|  e_{i_1} \wedge \cdots \wedge e_{i_n}  \right \| }
				\\ =& \sum\limits_{d \leq i_1 < \cdots < i_n, |\mathfrak{i}|=k}\frac{a_{{\mathfrak{i}}}}{n!}\sum\limits_{\pi  \in \Sigma_{n,\mathfrak{i}}^{\prime \prime} }\frac{ \beta_{i_{1},l_{\pi  (1)}} \cdots\beta_{i_{n},l_{\pi  (n)}}e_{i_1+l_{\pi(1)}} \wedge \cdots \wedge e_{i_n+l_{\pi(n)}}}{\left \|  e_{i_1} \wedge \cdots \wedge e_{i_n}  \right \| }
				\\=&
\sum\limits_{d \leq i_1 < \cdots < i_n, |\mathfrak{i}|=k}\frac{a_{{\mathfrak{i}}}}{n!}\sum\limits_{\pi  \in \Sigma_{n,\mathfrak{i}}^{\prime \prime\prime} }\frac{ \beta_{i_{1},l_{\pi  (1)}} \cdots\beta_{i_{n},l_{\pi  (n)}}e_{i_1+l_{\pi(1)}} \wedge \cdots \wedge e_{i_n+l_{\pi(n)}}}{\left \|  e_{i_1} \wedge \cdots \wedge e_{i_n}  \right \| }
\\  & ~ +\sum\limits_{t=1}^{\mathscr{M}}\sum\limits_{ \mathfrak{j} \in \mathcal{E}_{k+S_{\mathfrak{l}},t}^{\prime} }
				\left ( \sum\limits_{d \leq i_1 < \cdots < i_n, |\mathfrak{i}|=k}
				\sum\limits_{\pi \in \mathcal{N}_{\mathfrak{j}, \mathfrak{i}  } }\frac{ a_{\mathfrak{i}}		\beta_{i_{1},l_{\pi  (1)}} \cdots\beta_{i_{n},l_{\pi  (n)}}}{n!\left \|e_{i_1} \wedge \cdots \wedge e_{i_n} \right \|} \right )
				e_{j_1} \wedge \cdots \wedge e_{j_n}  
				\\ =& \sum\limits_{d \leq i_1 < \cdots < i_n, |\mathfrak{i}|=k}\frac{a_{{\mathfrak{i}}}}{n!}\sum\limits_{\pi  \in \Sigma_{n,\mathfrak{i}}^{\prime \prime\prime} }\frac{ \beta_{i_{1},l_{\pi  (1)}} \cdots\beta_{i_{n},l_{\pi  (n)}}e_{i_1+l_{\pi(1)}} \wedge \cdots \wedge e_{i_n+l_{\pi(n)}}}{\left \|  e_{i_1} \wedge \cdots \wedge e_{i_n}  \right \| }\\& ~ +
 \sum\limits_{\mathfrak{j}\in \tilde{\mathcal{A}}_{k+S_{\mathfrak{l}},\mathscr{M},\varepsilon }^{\prime}}
				\left ( \sum\limits_{d \leq i_1 < \cdots < i_n, |\mathfrak{i}|=k}
				\sum\limits_{\pi \in \mathcal{N}_{\mathfrak{j}, \mathfrak{i}  } }\frac{a_{\mathfrak{i}}				\beta_{i_{1},l_{\pi  (1)}}\cdots\beta_{i_{n},l_{\pi  (n)}}}{n!\left \| e_{i_1} \wedge \cdots \wedge e_{i_n} \right \|} \right )
				e_{j_1} \wedge \cdots \wedge e_{j_n}
				\\& ~ +\sum\limits_{\mathfrak{j}\in\check{\mathcal{A}}  _{k+S_{\mathfrak{l}},\mathscr{M},\varepsilon }^{\prime}  }      \left (\sum\limits_{d \leq i_1 < \cdots < i_n, |\mathfrak{i}|=k}
				\sum\limits_{\pi \in \mathcal{N}_{\mathfrak{j}, \mathfrak{i}  } }\frac{a_{\mathfrak{i}}		 	\beta_{i_{1},l_{\pi  (1)}} \cdots\beta_{i_{n},l_{\pi  (n)}}}{n!\left \| e_{i_1} \wedge \cdots \wedge e_{i_n} \right \|} \right )
				e_{j_1} \wedge \cdots \wedge e_{j_n}
				\\  & ~ + \sum\limits_{t=1}^{\mathscr{M}-1}\sum\limits_{\mathfrak{j} \in \mathcal{E}_{k+S_{\mathfrak{l}},t}^{\prime}  }
				\left (\sum\limits_{d \leq i_1 < \cdots < i_n, |\mathfrak{i}|=k}
				\sum\limits_{\pi \in \mathcal{N}_{\mathfrak{j}, \mathfrak{i}  } }\frac{a_{\mathfrak{i}}				\beta_{i_{1},l_{\pi  (1)}} \cdots\beta_{i_{n},l_{\pi  (n)}}}{n!\left \| e_{i_1} \wedge \cdots \wedge e_{i_n} \right \|} \right )
				e_{j_1} \wedge \cdots \wedge e_{j_n}.
			\end{aligned}
		\end{equation}
From Lemma \ref{lemma3.2} and the definitions of $\Sigma_{n,\mathfrak{i}}^{\prime \prime\prime},$ $\tilde{\mathcal{A}}_{k+S_\mathfrak{l},\mathscr{M},\varepsilon }^{\prime},$ $\check{\mathcal{A}}  _{k+S_\mathfrak{l},\mathscr{M},\varepsilon }^{\prime}$ and $\mathcal{E}_{k+S_\mathfrak{l},t }^{\prime},$ the last four terms in the above expression are mutually orthogonal.
		Moreover, for $\mathfrak{j} \in \mathcal{E}_{k+S_{\mathfrak{l}},\mathscr{M}} $, by (\ref{eqA3}), it is easy to see that
		\begin{equation}\label{eqA25}
			\left \| e_{j_1} \wedge \cdots \wedge e_{j_n} \right \| =\frac{1}{\sqrt{n!} }.
		\end{equation}
	It follows from (\ref{eq324}) and (\ref{eqA25}) that 
		\begin{equation}\label{eq2511}  \begin{aligned}
				&	\left \| \left (S_{{\alpha,l_{1}}}\wedge S_{{\alpha,l_{2}}}\wedge \cdots \wedge S_{{\alpha,l_{n}}}  \right )
				\sum\limits_{d \leq i_1 < \cdots < i_n, |\mathfrak{i}|=k}a_{ \mathfrak{i}} \frac{ e_{i_1} \wedge \cdots \wedge e_{i_n} }{\left \|  e_{i_1} \wedge \cdots \wedge e_{i_n}  \right \| }       \right \| ^{2}
				\\\ge& \left \| \sum\limits_{\mathfrak{j}\in \tilde{\mathcal{A}}_{k+S_{\mathfrak{l}},\mathscr{M},\varepsilon }^{\prime}}
				\left ( \sum\limits_{d \leq i_1 < \cdots < i_n, |\mathfrak{i}|=k}
				\sum\limits_{\pi \in \mathcal{N}_{\mathfrak{j}, \mathfrak{i}  } } \frac{a_{ \mathfrak{i}}		\beta_{i_{1},l_{\pi  (1)}}\cdots\beta_{i_{n},l_{\pi  (n)}}}{n!\left \| e_{i_1} \wedge \cdots \wedge e_{i_n} \right \|} \right )
				e_{j_1} \wedge \cdots \wedge e_{j_n}   \right \| ^{2}
				\\= &\sum\limits_{\mathfrak{j}\in \tilde{\mathcal{A}}_{k+S_{\mathfrak{l}},\mathscr{M},\varepsilon }^{\prime}} \left( \sum\limits_{d \leq i_1 < \cdots < i_n, |\mathfrak{i}|=k}
				\sum\limits_{\pi \in \mathcal{N}_{\mathfrak{j}, \mathfrak{i}  } }\frac{a_{ \mathfrak{i}}		\beta_{i_{1},l_{\pi  (1)}} \cdots\beta_{i_{n},l_{\pi  (n)}} }{n!\left \| e_{i_1} \wedge \cdots \wedge e_{i_n} \right \|} \right ) ^{2}
				\left \| e_{j_1} \wedge \cdots \wedge e_{j_n}  \right \| ^{2}
				\\= &\frac{1}{n!}\sum\limits_{\mathfrak{j}\in \tilde{\mathcal{A}}_{k+S_{\mathfrak{l}},\mathscr{M},\varepsilon }^{\prime}}
				\left( \sum\limits_{d \leq i_1 < \cdots < i_n, |\mathfrak{i}|=k}
				\sum\limits_{\pi \in \mathcal{N}_{\mathfrak{j}, \mathfrak{i}  } }\frac{a_{ \mathfrak{i}}		\beta_{i_{1},l_{\pi  (1)}} \cdots\beta_{i_{n},l_{\pi  (n)}} }{n!\left \| e_{i_1} \wedge \cdots \wedge e_{i_n} \right \|} \right ) ^{2}\\ \ge&  \frac{^{\#  } \tilde{\mathcal{A}}_{k+S_{\mathfrak{l}},\mathscr{M},\varepsilon }^{\prime} }{W_{(k,n)} } \left ( 1 - \varepsilon \right ) ^{2\left ( \left | l_{1} \right |  + \cdots +\left | l_{n} \right | \right ) },
		\end{aligned}	\end{equation}
		where the last inequality comes from \eqref{eqA34} and Lemma \ref{lemA248}, together with the fact that $ \mathfrak{i} \in \mathcal{R}_{\mathfrak{j} }$ if and only if $\mathcal{N}_{\mathfrak{j},\mathfrak{i} }  \neq \emptyset.$
		Thus, by Lemma \ref{lemA2.7}, \eqref{eq2511}, 
		 and the arbitrariness of $\varepsilon$, we conclude that \eqref{eqA20case} holds.
\qed

\section{Lower bound of the norm for symmetric tensor product}

This section is devoted to  solving \cite[Problem 1 and Problem 2]{GA}. At first, recall that \cite[Problem 1 and Problem 2]{GA} are stated  as follows.
\begin{prob}(\cite[Problem 1]{GA})
For $x_{1},x_{2},\cdots x_{n} \in \mathcal{H}$, 	is		$$\frac{1}{\sqrt{n!} } \left \| x_{1}    \right \| \left \| x_{2}     \right \|\cdots \left \| x_{n}     \right \| \le \left \| x_{1}\odot x_{2}\odot \cdots \odot x_{n}  \right \|? $$
	
\end{prob}
\begin{prob}(\cite[Problem 2]{GA})
	For $\mathbb{A}_{1},\mathbb{A}_{2},\cdots,\mathbb{A}_{n} \in \mathcal{ B(H)}$, is
	\begin{center}
		$\frac{1}{\sqrt{n!} } \sup\limits_{x\in H \atop \left \| x \right \| =1} \left \{ \left \| \mathbb{A}_{1}x \right \| \left \| \mathbb{A}_{2}x \right \| \cdots \left \| \mathbb{A}_{n}x \right \|  \right \} \le\left \| \mathbb{A}_{1}\odot \mathbb{A}_{2}\odot \cdots \odot \mathbb{A}_{n}    \right \|  ?$
	\end{center}
\end{prob}
The following two propositions answer the above problems affirmatively.
\begin{prop}\label{prop3.1}
	For $x_{1},x_{2},\cdots x_{n} \in \mathcal{H} $,
	\begin{center}
		$\frac{1}{\sqrt{n!} } \left \| x_{1}    \right \| \left \| x_{2}     \right \|\cdots \left \| x_{n}     \right \| \le \left \| x_{1}\odot x_{2}\odot \cdots \odot x_{n}  \right \| $.
	\end{center}
\end{prop}
\begin{proof}
	Let $\left \{ e_{i} \right \} _{i=1}^{\infty } $ be an orthonormal basis for $\mathcal{H}$. For every $i=1,2,\cdots,n$, write $x_{i}=\sum\limits_{j=0}^{\infty } a_{i,j} e_{j} $, where $\sum\limits_{j=0}^{\infty } \left | a_{i,j} \right | ^{2} <\infty $. Obviously
	\begin{center}
		$\left \| x_{i}  \right \| ^{2} =\sum\limits_{j=0}^{\infty } \left | a_{i,j} \right | ^{2} $.	
	\end{center}
	Since
	$$  \begin{aligned}
		\left \| x_{1}\odot x_{2}\odot \cdots \odot x_{n}  \right \|^{2}  & = \left \| \sum\limits_{j=0}^{\infty } a_{1,j}e_{j} \odot \sum\limits_{j=0}^{\infty } a_{2,j}e_{j} \odot \cdots \odot \sum\limits_{j=0}^{\infty } a_{n,j}e_{j}  \right \| ^{2}
		\\   & = \left \| \sum\limits_{i_{1}=0}^{\infty} \sum\limits_{i_{2}=0}^{\infty}\cdots \sum\limits_{i_{n}=0}^{\infty} a_{1,i_{1}} a_{2,i_{2}} \cdots a_{n,i_{n}} e_{i_{1} }\odot e_{i_{2} } \odot \cdots \odot e_{i_{n} }   \right \| ^{2}
		\\   & = \sum_{i_{1},i_{2},\cdots,i_{n} \ge 0 }\left | a_{1,i_{1} } a_{2,i_{2} }\cdots a_{n,i_{n} }\right | ^{2}\left \| e_{i_{1} }\odot e_{i_{2} } \odot \cdots \odot e_{i_{n} } \right \| ^{2}
	\end{aligned}$$
	and
	\begin{center}
		$\left \| e_{i_{1} }\odot e_{i_{2} } \odot \cdots \odot e_{i_{n} } \right \|^{2}  \ge \frac{1}{n!} $,
	\end{center}
	we have
	\begin{equation*}
			\left \| x_{1}\odot x_{2}\odot \cdots \odot x_{n}  \right \|^{2}   \ge \frac{1}{n!}\prod\limits_{j=1}^{n}  \sum\limits_{i_{j}=0}^{\infty} \left | a_{j,i_{j} }  \right | ^{2}  = \frac{1}{n!}\prod\limits_{j=1}^{n} \left \| x_{j}    \right \|^{2}. \qedhere
	\end{equation*}
\end{proof}
\begin{prop}
	For $\mathbb{A}_{1},\mathbb{A}_{2},\cdots,\mathbb{A}_{n}
	\in \mathcal{ B(H)}$,
	\begin{equation*}\label{eqsharp}
		\frac{1}{\sqrt{n!} } \sup\limits_{x\in \mathcal{H} \atop \left \| x \right \| =1} \left \{ \left \| \mathbb{A}_{1}x \right \| \left \| \mathbb{A}_{2}x \right \| \cdots \left \| \mathbb{A}_{n}x \right \|  \right \} \le\left \| \mathbb{A}_{1}\odot \mathbb{A}_{2}\odot \cdots \odot \mathbb{A}_{n}    \right \|  .
	\end{equation*}
\end{prop}

\begin{proof}
	By Propositon \ref{prop3.1},
		\begin{align*}
		\left \| \mathbb{A}_{1}\odot \mathbb{A}_{2}\odot \cdots \odot \mathbb{A}_{n}    \right \|
		& \ge  \sup_{x\in \mathcal{H} \atop \left \| x \right \| =1}\left \| \left ( \mathbb{A}_{1}\odot \mathbb{A}_{2}\odot \cdots \odot \mathbb{A}_{n}  \right ) \left ( x\otimes x\otimes \cdots \otimes x  \right )  \right \|
		\\ & =\sup_{x\in \mathcal{H} \atop \left \| x \right \| =1}\left \| \mathbb{A}_{1}x\odot \mathbb{A}_{2}x\odot \cdots \odot \mathbb{A}_{n}x   \right \|
		\\ & \ge \frac{1}{\sqrt{n!} } \sup_{x\in \mathcal{H} \atop \left \| x \right \| =1} \left \{ \left \| \mathbb{A}_{1}x \right \| \left \| \mathbb{A}_{2}x \right \| \cdots \left \| \mathbb{A}_{n}x \right \|  \right \}.  \qedhere
	\end{align*}
\end{proof}

\noindent\textbf{Acknowledgement.} This work is supported by National Science Foundation of China (No. 11871308, 12271298) and the Taishan Scholars Project (TSTP20250702). The authors thank the referee for helpful suggestions, which make this paper more readable.

  \noindent{Xiance Tian, School of Mathematics, Shandong University, Jinan 250100, Shandong, P. R. China, Email: 202311798@mail.sdu.edu.cn}

  \noindent{Penghui Wang, School of Mathematics, Shandong University, Jinan 250100, Shandong, P. R. China, Email: phwang@sdu.edu.cn}

  \noindent{Zeyou Zhu, School of Mathematics, Shandong University, Jinan 250100, Shandong, P. R. China, Email: 13155486329@163.com}

\end{document}